%% file: main.tex
\input{preamble}

\input{macros}

\begin{document}

\title{Moments of random multiplicative functions with polynomial coefficients}
\author{Xinyu Wang}
\maketitle

\begin{abstract}
Let \(f\) be a Steinhaus random multiplicative function and let \(g\) be a polynomial of degree \(d\). Write
\[
S_N=\frac1{\sqrt N}\sum_{n\le N}f(n)\,e(g(n)).
\]
We prove a quantitative dichotomy for the moments of \(S_N\): for each integer \(s\ge 2\), either \(\mathbb{E}|S_N|^{2s}\) is close to the Gaussian moment \(s!\), or the coefficients of \(g\) can be approximated by rationals with a small denominator. In particular, if the coefficients of \(g\) satisfy a Diophantine condition, then \(S_N\) converges in law to a complex normal distribution with mean \(0\) and variance \(1\). Lean~4 code for the proofs is provided, for convenience of verification.
\end{abstract}

\input{introduction}
\input{main-theorem}

\bibliographystyle{abbrv}
\bibliography{b}

\end{document}

%% file: preamble.tex
\documentclass[11pt]{article}
\usepackage[T1]{fontenc}
\usepackage[utf8]{inputenc}
\usepackage{color,xcolor}
\usepackage{mathpazo} % Palatino: elegant serif for text and math
\usepackage{microtype}
\usepackage{bbm}
\usepackage{amsmath}
\usepackage{amssymb}
\usepackage{amsthm}
\usepackage{verbatim}
\usepackage{fancyhdr}
\usepackage[margin=1in]{geometry}
\usepackage{mathtools}
\usepackage{mathrsfs}
\usepackage{tocloft}
\usepackage{graphicx}
\usepackage{tikz}
\usetikzlibrary{arrows.meta,positioning,fit,calc}
\usepackage{multicol}
\usepackage{abstract}
\definecolor{linkred}{RGB}{140,30,30}
\definecolor{inkplum}{RGB}{58,36,51}
\definecolor{planeivory}{RGB}{248,244,236}
\definecolor{plum}{RGB}{158,32,146}
\definecolor{gold}{RGB}{218,162,18}
\definecolor{goldink}{RGB}{148,96,0}
\colorlet{figink}{inkplum}
\colorlet{figplane}{planeivory}
\colorlet{figblue}{plum}
\colorlet{figbluepale}{plum!46!white}
\colorlet{figclay}{gold}
\colorlet{figrose}{gold}
\colorlet{goldpale}{gold!42!white}
\colorlet{figcell}{plum!62!white}
\colorlet{figcellpale}{plum!28!white}
\usepackage[colorlinks=true, pdfstartview=FitV, linkcolor=linkred,
citecolor=linkred, urlcolor=linkred]{hyperref}

\newtheorem{theorem}{Theorem}
\newtheorem{lemma}[theorem]{Lemma}
\newtheorem{proposition}[theorem]{Proposition}
\newtheorem{remark}[theorem]{Remark}
\newtheorem{definition}[theorem]{Definition}

\newtheorem{coro}[theorem]{Corollary}
\newtheorem{clm}{Claim}

%% file: macros.tex
\newcommand{\Sg}[1]{S_g\!\left(#1\right)}
\newcommand{\gprod}{%
  \prod_{i=1}^s e(g(n_i))\,\overline{e(g(m_i))}%
}

%% file: introduction.tex
% !TEX root = ../main.tex
\section{Introduction}
\label{sec:intro}

  A
Steinhaus random multiplicative function is a probabilistic model 
for the
M\"obius function.
The definition is as follows.
Let $\{f(p)\}_{p\text{ prime}}$ be a sequence of independent random variables,
identically uniformly distributed on the complex unit circle.
We then extend the definition of $f$ to all natural numbers by complete multiplicativity.
That is, writing $p^k\parallel n$ to mean that $p^k$ is the highest power of $p$ that divides $n$, we have
\[
f(n):=\prod_{p^k\parallel n}f(p)^k.
\]
As in the definition, the sequence $f(n)$ exhibits both randomness and multiplicativity.
Hence one may ask how the random multiplicative structure of $f(n)$ interacts with the structure of a deterministic sequence?
The sums of interest are $\sum_{n\le N}f(n)w(n)$, for a deterministic
sequence $w$.
In this paper, we take $w(n)=e(g(n))$, where $e(x)=\exp(2\pi i x)$
and
$g$ is a real polynomial of degree~$d$,
\[
g(n)=\sum_{j=0}^d\beta_j n^j,
\]
where $\beta_j\in\mathbb{R}$, $j=0,1,\ldots,d$ are the coefficients of the polynomial.
%%% Explain what beta is. and how to project into R/Z
Write
\[
S_N=\frac1{\sqrt N}\sum_{n\le N}f(n)\,e(g(n)).
\]
Here $\sqrt N$ is the size of the standard deviation of
$\sum_{n\le N}f(n)\,e(g(n))$.  

%Note that $S_N$ is a random variable. 
We are interested in the moments of the random variables $S_N$.
We give a quantitative dichotomy for the moments of $S_N$ --- the
diophantine property of the coefficients $\beta_i$ determines the moment behaviour.

We adapt the definition 
of the smoothness norm as in \cite{BT} to describe the Diophantine property.
Throughout, $\|\cdot\|_{\mathbb{R}/\mathbb{Z}}$ denotes distance to the nearest integer.

\begin{definition}[{$C^\infty[N]$ norm}]
\label{def:c-infty}
For $g(n)=\sum_{j=0}^d\beta_j n^j$ as above and $N\geq 1$, we use the
monomial-coefficient smoothness norm
\[
\|g\|_{C^\infty[N]}
:=\max_{1\leq j\leq d}\, N^j\,\|\beta_j\|_{\mathbb{R}/\mathbb{Z}}.
\]
\end{definition}
As in Green--Tao \cite[Definition~2.7]{BT}, this smoothness norm is designed to capture the notion of a polynomial sequence which is slowly-varying.
Note that an upper bound on $\|qg\|_{C^\infty[N]}$ for some small nonzero integer $q$ says that the coefficients of $g$ can be approximated by rationals with a small denominator.

The following theorem is our main result.
\begin{theorem}[Moment dichotomy]\label{thm:main}
Given integers $d \geq 1$ and $s\geq 2$, there exists a constant $c_0=c_0(d,s)>0$ such that the following holds.
For all integers $N$ large enough depending only on $d$ and~$s$, let $S_N$ be defined as above for any polynomial $g$ of degree~$d$.
For every $N^{-c_0}<\delta<1/8$, at least one of the following holds:
\begin{enumerate}
    \item \[
    \Big|\mathbb{E}\big|S_N\big|^{2s}-s!\Big|
    \ll_{d,s} \delta(\log N)^{O(s^2)}.
    \]
    \item There exists $q\in\mathbb{Z}\setminus\{0\}$ with $|q|\ll_d\delta^{-O_d(1)}$ such that
    \[
    \|qg\|_{C^{\infty}[N]}\ll_d\delta^{-O_d(1)}.
    \]
    \end{enumerate}
\end{theorem}
Note that $s!$ is exactly the $2s$-th Gaussian moment of a complex normal distribution with mean $0$ and variance $1$.
In other words, either the $2s$-th moment is close to $s!$, or $qg$ is slowly varying for some small nonzero integer $q$.

Let us note that allowing the second alternative is necessary.
If we take $g$ to be a constant, then
$|S_N|=|N^{-1/2}\sum_{n\le N}f(n)|$, and the $2s$-th moment of this sum
is of size $(\log N)^{\Theta(s^2)}$ rather than $s!$, by Harper,
Nikeghbali and Radziwi{\l}{\l} \cite{HNR15} and Heap and Lindqvist
\cite{HeapLindqvist16}. 

When $s=1$ one has $\mathbb{E}|S_N|^2=1$ immediately, by Steinhaus orthogonality, that is,
since $\mathbb{E} f(p)=0$ and natural numbers have unique prime factorisations, it is immediate that for all natural numbers $n$ and $m$,
\[
\mathbb{E}\bigl[f(n)\overline{f(m)}\bigr]=\mathbf{1}_{n=m}.
\]
Let us also note that the odd moments of $S_N$ tend to zero as $N\to\infty$, which can be 
shown by easy arguments. See the arguments in the proof of Corollary \ref{thm:cor-clt}.

For related results regarding the moments of partial sums of random multiplicative functions, Pandey, Wang and Xu \cite{PWX24} showed
that the $2s$-th moments of $H^{-1/2}\sum_{x<n\le x+H}f(n)$ match the Gaussian
values $s!$ when $H$ is a sufficiently short interval.  Wang and Xu
\cite{WX24} showed that the $2s$-th moments of
$N^{-1/2}\sum_{n\le N}f(P(n))$ match the Gaussian values $s!$, for a
polynomial $P\in\mathbb{Z}[x]$ with at least two distinct complex roots.
Meanwhile, Benatar, Nishry and Rodgers \cite{BNR} showed that
\[
\mathbb{E}\Biggl|\int_0^1\Bigl|N^{-1/2}\sum_{n\le N}f(n)\,e(n\theta)\Bigr|^{2s}\,d\theta-s!\Biggr|^2\to 0,
\]
where the expectation is over \(f\).
This holds for $s\ll(\log N/\log\log N)^{1/3}$.
Each of the results above yields a central limit theorem.

If the coefficients of $g$ satisfy a Diophantine condition, then the following central limit theorem holds.
\begin{coro}[Central limit theorem]\label{thm:cor-clt}
Write $g(n)=\beta_d n^d+\cdots+\beta_1 n+\beta_0$. Assume that for every $\varepsilon>0$ there exists $C=C(g,\varepsilon)>0$ such that for every nonzero integer $q$,
\[
\max_{1\leq i\leq d}\,\|q\beta_i\|_{\mathbb{R}/\mathbb{Z}}
\geq C\exp\bigl(-|q|^{\varepsilon}\bigr).
\]
Then $S_N$ converges in law to a complex normal distribution with mean $0$ and variance $1$.
\end{coro}

For example, the hypothesis holds if some coefficient \(\beta_i\) is an algebraic irrational. In fact, it can be shown that the exceptional set has Hausdorff dimension zero.

From Harper \cite{HarperI} it follows that Corollary~\ref{thm:cor-clt} fails for rational phases, so some Diophantine condition is necessary.
Previously, Soundararajan and Xu \cite{KX} gave a general criterion on a deterministic weight $w$ under which $N^{-1/2}\sum_{n\le N}f(n)w(n)$ converges in law to a centred complex normal of variance $1$. Their proof uses the martingale central limit theorem and a fourth-moment computation. For the linear phase $w(n)=e(n\theta)$, \cite[Theorem~1.6]{KX} imposes a condition of the same shape with a single value $\varepsilon=1/50$.
Here, however, the Diophantine condition is assumed for every $\varepsilon>0$. This is because we use the moment method to prove the central limit theorem, and $\varepsilon$ must be taken smaller as the moment order increases
in order to make sure the higher moments of $S_N$ tend to Gaussian moments. But in fact, for the fourth-moment, we only need to assume the Diophantine condition for a fixed $\varepsilon$ depending only on the degree of the polynomial.

The implied constants in Theorem~\ref{thm:main} depend on $s$, and we state the result for each fixed $s$. The same argument should allow $s$ to grow slowly with $N$, once its dependence is tracked uniformly. For instance, one expects this for $s=o\bigl((\log N/\log\log N)^{1/3}\bigr)$.

%%% Please revise English grammar : the use of articles.
We believe that Theorem \ref{thm:main} can be extended to other sequences, for example
nilsequences such as $w(n)=e(\{n\alpha\}n\beta)$ where $\alpha,\beta\in \mathbb{R}/\mathbb{Z}$, and $\{n\alpha\}$ is the fractional part of $n\alpha$.
But the central limit theorem need not hold for an arbitrary Lipschitz function $F$ in place of $e(\cdot)$.
See Schlitt \cite{Schlitt26} for a necessary and sufficient condition
on \(F\) under which \(\sum_{n\le N}f(n)F(n\alpha)\) obeys a central limit
theorem, where \(F\) is \(1\)-periodic of bounded variation and \(\alpha\)
is irrational.

The main input of Theorem \ref{thm:main} is the quantitative Weyl theorem of Green--Tao \cite{BT} for polynomial orbits on the circle.
With this main input, the other ingredients
include the inclusion-exclusion decomposition, some parameterisation lemmas, divisor sum estimates and effective recurrence arguments that appeared
in \cite{BT} and \cite[\S3]{GTmob}.
The proof strategy is further explained in Section~\ref{sec:proof-strategy}.

\subsection{Formalisation}
\label{sec:formalisation}
For convenience of verification, we provide Lean~4 code for the proofs
of the results at
\begin{center}
\url{https://github.com/cutedonkeyhhh/RMF_repo}.
\end{center}
See this repository for the details.

\subsection{Notation}
\label{sec:notation}

We write $f\ll g$, or equivalently $f=O(g)$, if $|f|\leq C|g|$ for an absolute constant $C$.
Subscripts indicate the dependence of implied constants on parameters.

The $k$-fold divisor function is denoted $\tau_k$.
If $\mathcal{B}$ is a set and $\mathcal{P}$ is a property, we write $\mathcal{B}\cap\{\mathcal{P}\}$ for the subset of $\mathcal{B}$ consisting of those elements that satisfy $\mathcal{P}$.

\subsection*{Acknowledgements}

The author is very grateful to  Weikun He for discussions, corrections, and suggestions that greatly improved the presentation of this paper. Sincere thanks also go to  Changguang Dong for very helpful comments. 
The author is supported by the National Key R\&D Program of China (No. 2022YFA1007500).
Grok~4.5 and Grok~4.6 were used to provide Lean code and to polish the writing.

%% file: main-theorem.tex
% !TEX root = ../main.tex
\section{Setup and inductive statement}
\label{sec:setup}

The proof of Theorem~\ref{thm:main} is by induction on~$s$.
We begin with the divisor bounds, which will be used throughout the paper, all of which are standard.
\begin{lemma}[Divisor sums]\label{lem:divisor}
For integers $N\geq 3$, $l\geq 1$, and the $l$-fold divisor function $\tau_l$, we have
\begin{gather*}
\sum_{n\leq N}\tau_l(n)\leq N(1+\log N)^{l-1}\leq N(2\log N)^{l-1},\\
\sum_{n\leq N}\tau_l(n)^2\leq N(2\log N)^{l^2-1},
\shortintertext{and for any positive integer $A<N$,}
\sum_{A\leq n\leq N}\frac{\tau_l(n)^2}{n^2}\leq \frac{5(2\log N)^{l^2-1}}{A}.
\shortintertext{Moreover,}
\sum_{n\leq N}\frac{\tau_l(n)}{n}\leq (2\log N)^{l}.
\end{gather*}
\end{lemma}
\begin{proof}
The first inequality is \cite[Lemma~3.1]{BNR}, and the second is \cite[(3.4)]{BNR}.
The third follows from the second by summation by parts.
The last is
\(\sum_{n\leq N}\tau_l(n)/n\leq\bigl(\sum_{n\leq N}n^{-1}\bigr)^l\leq(1+\log N)^l\leq(2\log N)^l\).
\end{proof}

\subsection{Normalised and unnormalised moments}
We now consider how to compute the $2s$-th moment of the partial sum $S_N$.
Set
\[
U_s(N)=\mathbb{E}\Big|\sum_{n\leq N}f(n)\,e(g(n))\Big|^{2s},
\qquad
\mathcal{M}_s(N)=N^{-s}U_s(N)=\mathbb{E}|S_N|^{2s},
\]
where $\mathcal{M}_s(N)$ is the normalised moment.
It is a natural idea to transform the problem of computing the $2s$-th moment of $S_N$ to the computation of the exponential sum over some structured set, which is a purely arithmetic problem.
For $N\geq 1$, set
\[
V=\{(n_1,\ldots,n_s,m_1,\ldots,m_s)\in [1,N]^{2s}\cap\mathbb{Z}^{2s} : n_1\cdots n_s=m_1\cdots m_s\}.
\]
This set $V$ is the playground. Moreover, we need the following notation.
  For any subset $\mathcal{F}\subseteq V$, denote by $\Sg{\mathcal{F}}$ the sum of the exponential function over $\mathcal{F}$, that is,
\[
\Sg{\mathcal{F}}
=\sum_{(\vec{n}_s,\vec{m}_s)\in\mathcal{F}}\ \gprod.
\]
With the above notation, we have the following moment formula.
\begin{lemma}[Moment formula]\label{lem:moment-formula}
$U_s(N)=\Sg{V}$.
\end{lemma}
\begin{proof}
This identity is a direct consequence of the orthogonality of $f$, that is, for all natural numbers $n$ and $m$,
\[
\mathbb{E}\bigl[f(n)\overline{f(m)}\bigr]=\mathbf{1}_{n=m}.
\]
\end{proof}

This lemma says that the computation of the $2s$th moment of the partial sum 
can be reduced to the computation of 
the exponential sum over the set $V$. Hence the structure of the set $V$ is crucial. We will study the structure of $V$ in the next section.

We prove Theorem~\ref{thm:main} by induction on~$s$.
When $s=1$, the condition defining $V$ is simply $n_1=m_1$. Lemma~\ref{lem:moment-formula} therefore gives $U_1(N)=N$ and $\mathcal{M}_1(N)=1$. Write $\mathcal{H}(1)$ for this identity. This is the base of the induction.
Next we give the inductive statement.
\subsection{The inductive statement}
For an integer $s\geq 2$, write $\mathcal{H}(s)$ for the following proposition.
For every integer $d\geq 1$, there exists a constant $c_0=c_0(d,s)>0$ such that the following holds.
For every polynomial $g$ of degree~$d$, for all integers $N$ large enough depending only on $d$ and~$s$, and for every $N^{-c_0}<\delta<1/8$, at least one of the following holds:
\begin{enumerate}
    \item \[
    \bigl|\mathcal{M}_s(N)-s!\bigr|
    \ll_{d,s} \delta(\log N)^{O(s^2)}.
    \]
    \item There exists $q\in\mathbb{Z}\setminus\{0\}$ with $|q|\ll_d\delta^{-O_d(1)}$ such that
    \[
    \|qg\|_{C^{\infty}[N]}\ll_d\delta^{-O_d(1)}.
    \]
\end{enumerate}
Thus $\mathcal{H}(s)$ is exactly the conclusion of Theorem~\ref{thm:main} at moment order~$2s$.  
\section{The inclusion--exclusion decomposition}
\label{sec:inclusion-exclusion}

From the previous section, we have reduced the problem of computing the $2s$th moment of the partial sum $S_N$ to the computation of the exponential sum over the set $V$.
We now describe the set $V$ and the inclusion--exclusion decomposition.
Firstly, it is easy to see that the solution set $V$ of the diophantine equation
\[
n_1\cdots n_s=m_1\cdots m_s
\]
contains the trivial solutions, in which \(\vec{m}_s\) is a rearrangement of \(\vec{n}_s\):
that is, \(n_i=m_{\sigma(i)}\) for \(i=1,\dots,s\), for some \(\sigma\in S_s\), which is the symmetric group on \(s\) letters.
On each such point, reindexing the product over \(m_i\) gives
\[
\prod_{i=1}^s e(g(n_i))\,\overline{e(g(m_i))}
=\prod_{i=1}^s e(g(n_i))\,\overline{e(g(m_{\sigma(i)}))}
=\prod_{i=1}^s e(g(n_i))\,\overline{e(g(n_i))}
=1,
\]
so each trivial solution contributes \(1\) to \(U_s(N)\).
It can be shown that the number of distinct trivial solutions is \(s!N^s+O_s(N^{s-1})\).
Since each contributes \(1\), this is the main term of \(U_s(N)\).

It is hard to directly compute the exponential sum over
the non-trivial solutions in $V$.
This forces us to study the structure of~$V$.

The set $V$ exhibits a lot of symmetry,
since it is invariant under the action of the symmetric group $S_s$ on the indices $1,\dots,s$.
The first step of this work is to decompose the set $V$
into structured sets by making use of this symmetry.
We now describe this decomposition.

Let $A>0$ (to be chosen later).  For $i=1,\ldots,s$,
define
\[
\mathscr{F}_i=\{(\vec{n}_s,\vec{m}_s)\in V:\ \gcd(n_1,m_i)\geq A\}.
\]
We let \(\mathscr{F}_i\) implicitly depend on the parameter \(A\).
The above definition loosely captures the information of the greatest common divisors of $n_1$ and $m_i$ as $i$ varies.
The reason we require the greatest common divisor 
to be larger than $A$ is to guarantee that the effective recurrence argument of \cite[\S3]{GTmob} can be applied. 
Each $\mathscr{F}_i$ contains the solutions where $n_1=m_i$, which will 
contribute to the main term. The structure is easier when we consider the exponential sum over 
each $\mathscr{F}_i$ separately.
We are able to approximate the whole solution set $V$ by the union of the $\mathscr{F}_i$ up to a negligible sparse complement, which is the content of the following lemma.
\begin{lemma}[Sparse complement]\label{lem:sparse-complement}
Let \(s\geq 2\) and \(N\geq 3\) be integers, and let \(A\) be a positive integer. With \(V\) and \(\mathscr{F}_i\) as above,
\[
\#(V\setminus\cup_{i=1}^s\mathscr{F}_i)\leq N^{s-1}A^{2s}(2s\log N)^{(s+1)^2}.
\]
\end{lemma}

\begin{proof}
If $(n_1,\ldots,n_s,m_1,\ldots,m_s)\in V\setminus\cup_i\mathscr{F}_i$, then $\gcd(n_1,m_j)<A$ for all $j=1,\ldots,s$.
Since $n_1\mid m_1\cdots m_s$, one has $n_1\mid\prod_{j=1}^s\gcd(n_1,m_j)$, hence $n_1\leq A^s$.

Fix $n_1=g$.  The equation $n_1\cdots n_s=m_1\cdots m_s$ is equivalent to $m_1\cdots m_s=gn$ for some $1\leq n\leq N^{s-1}$ (writing $n=n_2\cdots n_s$).
For each such $n$, the number of $(n_2,\ldots,n_s,m_1,\ldots,m_s)$ with $n_2\cdots n_s=n$ and $m_1\cdots m_s=gn$ is at most $\tau_s(gn)\tau_{s-1}(n)$.
Thus the fibre over $n_1=g$ has size at most
\[
\sum_{n\leq N^{s-1}}\tau_s(gn)\tau_{s-1}(n).
\]
By Cauchy--Schwarz and the square-sum bound in Lemma~\ref{lem:divisor},
together with \(\log(gN^{s-1})\leq s\log N\) for \(g\leq N\),
\[
\sum_{n\leq N^{s-1}}\tau_s(gn)\tau_{s-1}(n)
\leq gN^{s-1}(2s\log N)^{s^2}.
\]
Summing over $g\leq A^s$ and using \(s^2\leq(s+1)^2\) gives
\[
\#(V\setminus\cup_i\mathscr{F}_i)
\leq \sum_{g\leq A^s} gN^{s-1}(2s\log N)^{s^2}
\leq N^{s-1}A^{2s}(2s\log N)^{(s+1)^2}.
\]
\end{proof}

By inclusion--exclusion,
\begin{equation}\label{eq:IE-general}
\sum_{\bigcup_i\mathscr{F}_i}
= \sum_i\sum_{\mathscr{F}_i} - \sum_{i<j}\sum_{\mathscr{F}_i\cap\mathscr{F}_j} + \cdots
+ (-1)^{s-1}\sum_{\mathscr{F}_1\cap\cdots\cap\mathscr{F}_s}.
\end{equation}
\label{lem:inclusion-exclusion}
Hence summation over the union $\bigcup_i\mathscr{F}_i$ can be decomposed into the sum over the single sets $\mathscr{F}_i$ and the sum over the intersections $\mathscr{F}_{i_1}\cap\cdots\cap\mathscr{F}_{i_\kappa}$.
Next we explain the strategy to prove Theorem \ref{thm:main}.
\subsection{Proof strategy}
\label{sec:proof-strategy}
We prove Theorem~\ref{thm:main} by induction on \(s\geq 2\). Lemma~\ref{lem:sparse-complement} shows that the complement of \(\bigcup_i\mathscr{F}_i\) in \(V\) is negligible, so the sum over \(V\) reduces to the sum over the union. By the inclusion--exclusion above, this sum splits into the single sets \(\mathscr{F}_i\) and the intersections \(\mathscr{F}_{i_1}\cap\cdots\cap\mathscr{F}_{i_\kappa}\). In Section~\ref{sec:F1}, we estimate the single sets, and in Section~\ref{sec:intersection}, we estimate the intersections.

For both estimates, it is convenient to separate the diagonal and off-diagonal parts.
For a single set \(\mathscr{F}_i\), the diagonal part is the subset on which \(n_1=m_i\), and the off-diagonal part is
\[
\mathscr{F}_i^\circ:=\mathscr{F}_i\setminus\{n_1=m_i\}.
\]
For an intersection \(\mathscr{F}_{i_1}\cap\cdots\cap\mathscr{F}_{i_\kappa}\), the diagonal part is the subset on which \(n_1=m_{i_1}\), and the off-diagonal part is
\[
(\mathscr{F}_{i_1}\cap\cdots\cap\mathscr{F}_{i_\kappa})^\circ
:=(\mathscr{F}_{i_1}\cap\cdots\cap\mathscr{F}_{i_\kappa})\setminus\{n_1=m_{i_1}\}.
\]
Figure~\ref{fig:roadmap} records these two off-diagonal sets.

The main term comes from the diagonal part of each \(\mathscr{F}_i\), by the inductive hypothesis \(\mathcal{H}(s-1)\), and the rest contributes no main term. The diagonal part of each intersection is controlled by the divisor bound. The off-diagonal parts of both \(\mathscr{F}_i\) and the intersections are explained below, and this is the most technical part of the paper.

\begin{figure}[ht]
\centering
\resizebox{0.88\textwidth}{!}{%
\begin{tikzpicture}[
  font=\small,
  >=Stealth,
  arr/.style={->, thick, draw=figink},
  box/.style={draw=figink, rounded corners=2.5pt, fill=figplane},
]
  % ----- left: tighter frame; flower above a label that stays inside -----
  \coordinate (Vc) at (0,0.24);
  \node[box, minimum width=2.62cm, minimum height=2.96cm] (Vbox) at (0,0) {};
  \node[above=2pt, font=\normalsize\bfseries] at (Vbox.north) {$V$};
  \foreach \ang in {90,30,-30,-90,-150,150} {
    \fill[plum!32, opacity=0.62, rotate around={\ang:(Vc)}]
      ($(Vc)+(0,0.38)$) ellipse (0.38 and 0.62);
  }
  \foreach \ang in {90,30,-30,-90,-150,150} {
    \draw[figink!75, line width=0.55pt, rotate around={\ang:(Vc)}]
      ($(Vc)+(0,0.38)$) ellipse (0.38 and 0.62);
  }
  \fill[gold] (Vc) circle (0.22);
  \draw[figink!75, line width=0.55pt] (Vc) circle (0.22);
  \node[font=\tiny, anchor=south, inner sep=0pt] at (0,-1.34) {$\bigcup_i\mathscr{F}_i$};

  % ----- middle: one petal, and the intersection of three sets -----
  \node[box, minimum width=2.15cm, minimum height=1.42cm] (Fi) at (4.08,0.82) {};
  \fill[plum!34] ($(Fi.center)+(0,0.16)$) ellipse (0.32 and 0.42);
  \draw[figink!75, line width=0.55pt] ($(Fi.center)+(0,0.16)$) ellipse (0.32 and 0.42);
  \node[font=\tiny, anchor=north, inner sep=0pt] at ($(Fi.center)+(0,-0.42)$) {$\mathscr{F}_i^\circ$};
  \node[above=1pt, font=\scriptsize, inner sep=1pt] at (Fi.north) {off-diagonal};

  \node[box, minimum width=2.52cm, minimum height=1.72cm] (Fint) at (4.26,-0.92) {};
  \coordinate (Ic) at ($(Fint.center)+(0,0.14)$);
  \foreach \ang in {90,210,330} {
    \fill[plum!34]
      ($(Ic)+({0.22*cos(\ang)},{0.22*sin(\ang)})$) circle (0.40);
  }
  \begin{scope}
    \clip ($(Ic)+({0.22*cos(90)},{0.22*sin(90)})$) circle (0.40);
    \clip ($(Ic)+({0.22*cos(210)},{0.22*sin(210)})$) circle (0.40);
    \fill[gold] ($(Ic)+({0.22*cos(330)},{0.22*sin(330)})$) circle (0.40);
  \end{scope}
  \foreach \ang in {90,210,330} {
    \draw[figink!75, line width=0.5pt]
      ($(Ic)+({0.22*cos(\ang)},{0.22*sin(\ang)})$) circle (0.40);
  }
  \node[font=\tiny, anchor=north, inner sep=0pt]
    at ($(Fint.center)+(0,-0.58)$)
    {$(\mathscr{F}_{i_1}\cap\cdots\cap\mathscr{F}_{i_\kappa})^\circ$};

  \draw[arr] ($(Vbox.east)+(0,0.48)$) -- ++(0.48,0) |- (Fi.west);
  \draw[arr] ($(Vbox.east)+(0,-0.48)$) -- ++(0.48,0) |- (Fint.west);

  % ----- top-right: parameters -----
  \node[box, inner sep=4pt, align=center, minimum width=3.15cm]
    (Par) at (8.58,0.82)
    {$\mathbb{Z}_{>0}^{s^{2}}$\ \ or\ \ $\mathbb{Z}_{>0}^{3s}$};

  \draw[arr] (Fi.east) -- node[above=4pt, font=\tiny, inner sep=0.5pt]{parameterisation} (Par.west);

  % ----- bottom-right: good and bad fibres in space -----
  \node[box, minimum width=3.15cm, minimum height=2.05cm]
    (Fib) at (8.58,-1.32) {};
  \begin{scope}[shift={(Fib.center)}]
    \coordinate (O) at (-1.05,-0.72);
    \coordinate (U) at (1.95,0.15);
    \coordinate (V) at (-0.46,0.30);
    \def\ah{0.20}
    \fill[white] (O) -- ++(U) -- ++(V) -- ++($-1*(U)$) -- cycle;
    \draw[figink] (O) -- ++(U) -- ++(V) -- ++($-1*(U)$) -- cycle;
    \foreach \s/\t/\len in {0.14/0.72/0.42, 0.26/0.28/0.55, 0.36/0.55/0.48, 0.18/0.10/0.36} {
      \coordinate (P) at ($(O)+\s*(U)+\t*(V)$);
      \coordinate (Bot) at ($(P)+(0,\ah)$);
      \coordinate (Top) at ($(Bot)+(0,\len)$);
      \draw[figink!40, dotted] (P) -- (Bot);
      \draw[figbluepale, line width=1.35pt] ($(Bot)+(0.03,0)$) -- ($(Top)+(0.03,0)$);
      \draw[plum, line width=0.9pt] (Bot) -- (Top);
      \fill[plum] (Bot) circle (0.8pt);
      \fill[plum] (Top) circle (0.5pt);
    }
    \foreach \s/\t/\len in {0.58/0.70/0.40, 0.72/0.26/0.52, 0.84/0.52/0.44, 0.64/0.08/0.34} {
      \coordinate (P) at ($(O)+\s*(U)+\t*(V)$);
      \coordinate (Bot) at ($(P)+(0,\ah)$);
      \coordinate (Top) at ($(Bot)+(0,\len)$);
      \draw[figink!40, dotted] (P) -- (Bot);
      \draw[gold, line width=1.25pt, dash pattern=on 2.2pt off 1.05pt] (Bot) -- (Top);
      \fill[gold] (Bot) circle (0.8pt);
      \fill[gold] (Top) circle (0.5pt);
    }
    \node[font=\scriptsize, text=plum, anchor=south]
      at ($(O)+0.24*(U)+0.40*(V)+(0,\ah)+(0,0.66)$) {good};
    \node[font=\scriptsize, text=goldink, anchor=south]
      at ($(O)+0.74*(U)+0.36*(V)+(0,\ah)+(0,0.64)$) {bad};
  \end{scope}

  \coordinate (Pin) at ($(Par.south)+(-0.68,0)$);
  \coordinate (Gap) at ($(Fib.north)!0.5!(Par.south)$);
  \draw[arr] (Fint.east) -- ++(0.85,0) coordinate (Turn)
    -- (Turn |- Gap) -- (Pin |- Gap) -- (Pin);
  \draw[arr] (Par.south) -- node[right=2pt, font=\scriptsize, inner sep=1pt]{projection} (Fib.north);
\end{tikzpicture}%
}
\caption{The proof strategy}
\label{fig:roadmap}
\end{figure}
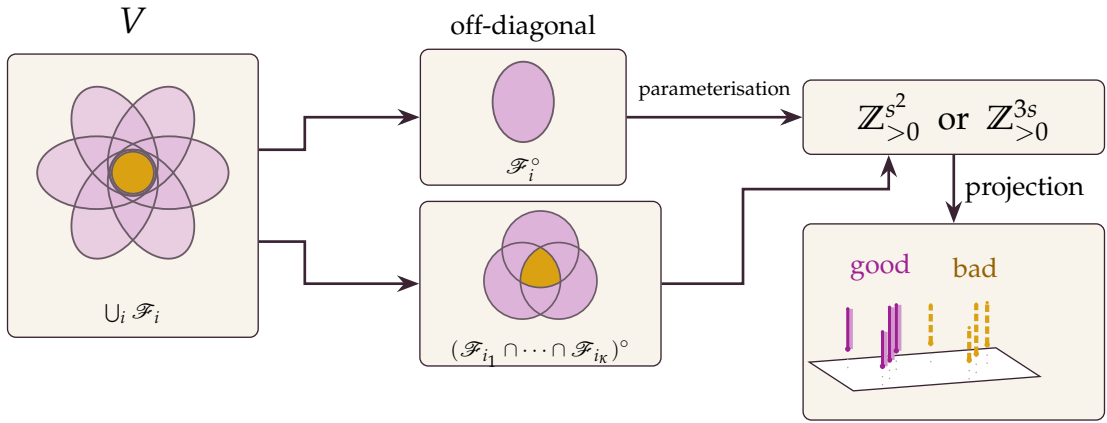
The two off-diagonal sets in Figure~\ref{fig:roadmap} are treated by the same sequence of steps. Parameterisation lemmas place the solutions in \(\mathbb{Z}_{>0}^{s^2}\) or in \(\mathbb{Z}_{>0}^{3s}\). A projection then forgets some of these coordinates, leaving a base together with a fibre. The base points are split into good and bad according to the size of the exponential sum over the fibre. 
The number of bad base points is controlled by
 the Diophantine dichotomy of Green--Tao, explained in the next section.
\section{Diophantine dichotomy}
\label{sec:diophantine}
This section explains the diophantine dichotomy result needed for the proofs.
We
prove Lemma~\ref{lem:diophantine-nil}  following the
argument of \cite[\S3]{GTmob}. The main input is the quantitative Leibman theorem of Green--Tao \cite[Theorem~2.9, Proposition~4.3]{BT}, specialised to polynomials on the circle, which is a quantitative Weyl theorem.

Our notion of smoothness norm is slightly different from the one used in Green--Tao.
They \cite[Definition~2.7]{BT} use the binomial basis and the norm
\[
g(n)=\sum_{j=0}^d \alpha_j\binom{n}{j},
\qquad
\|g\|_{C_{\mathrm{GT}}^\infty[N]}
:=\sup_{1\leq j\leq d}\, N^j\,\|\alpha_j\|_{\mathbb{R}/\mathbb{Z}}.
\]
By a change of basis, there is an integer $q_d$ depending only on $d$ such that
\[
\|q_d g\|_{C^\infty[N]} \ll_d \|g\|_{C_{\mathrm{GT}}^\infty[N]}.
\]
So the following statement does not change if we use the monomial basis instead.

\begin{theorem}[{Green--Tao \cite[Theorem~2.9, Proposition~4.3]{BT}, for polynomials on the circle}]
\label{prop:quantitative-weyl-circle}
\label{prop:gt-weyl}
Let \(p\colon\mathbb Z\to\mathbb R/\mathbb Z\) be a polynomial of degree at
most \(d\), let \(M\geq1\), and let \(0<\delta<1/2\).  Then at least one of
the following holds:
\begin{enumerate}
    \item The sequence \((p(n))_{n=1}^M\) is
    \(\delta\)-equidistributed in \(\mathbb R/\mathbb Z\), in the sense that
    \[
    \left|\frac1M\sum_{n=1}^M F(p(n))
    -\int_{\mathbb R/\mathbb Z}F\right|
    \leq\delta\|F\|_{\mathrm{Lip}}
    \]
    for every Lipschitz function
    \(F\colon\mathbb R/\mathbb Z\to\mathbb C\).
    \item There exists a nonzero integer \(k\) such that
    \[
    |k|\ll_d\delta^{-O_d(1)},
    \qquad
    \|kp\|_{C^\infty[M]}
    \ll_d\delta^{-O_d(1)}.
    \]
\end{enumerate}
\end{theorem}
This dichotomy says that, either the polynomial sequence
is equidistributed or the polynomial is non-diophantine. Moreover, this result is quantitative, making it convenient for our application.

Several more ingredients from their work are used.
We will also use the following effective recurrence lemma.
\begin{lemma}[{Effective recurrence, Green--Tao \cite[Lemma~3.2]{BT}}]
\label{lem:bt-recurrence}
Let \(\alpha\in\mathbb R\), let \(0<\sigma<1/2\), and let \(0<\varepsilon\le\sigma/2\).
Let \(I\subseteq\mathbb R/\mathbb Z\) be an interval of length \(\varepsilon\), and let \(N\ge 1\).
Suppose that \(\alpha n\in I\) for at least \(\sigma N\) integers \(n\) with \(1\le n\le N\).
Then there exists \(k\in\mathbb Z\) with \(0<|k|\ll\sigma^{-O(1)}\) such that
\[
\|k\alpha\|_{\mathbb R/\mathbb Z}\ll\varepsilon\,\sigma^{-O(1)}/N.
\]
\end{lemma}

We will also use the following Waring-type count.
\begin{lemma}[{Green--Tao \cite[Lemma~3.3]{GTmob}}]
\label{lem:gt-waring}
Let \(j\ge 1\) and \(K\ge 1\) be integers, let \(0<\eta\le 1\), and let
\(S\subseteq[K]\) with \(|S|=\eta K\).  Let \(t\ge 2^j+1\).  Then
\(\gg_{j,t}\eta^{2t}K^j\) integers in the interval \([tK^j]\) can be written
in the form \(k_1^j+\cdots+k_t^j\) with \(k_1,\ldots,k_t\in S\).
\end{lemma}

With these inputs, we prove the next lemma formulated in this way for various adaptations.
\begin{lemma}[Exponential dichotomy]\label{lem:diophantine-nil}
There exists a constant $C_d>5$, depending only on $d$, with the
following property. Let $g$ be a real polynomial of degree~$d$. Let $0<\delta<1/8$, let $0<D<N$, and put
$N_D=\lfloor N/D\rfloor$. Suppose that
\[
\frac ND>\delta^{-C_d},
\]
and let $L=[a,b]\subseteq[1,2N_D]$ be an interval of integers with
\[
\frac14\cdot\frac ND<|L|<2\cdot\frac ND
\qquad\text{and}\qquad
a<\delta^{-C_d-3}.
\]
Then at least one of the following alternatives holds:
\begin{enumerate}
    \item
    \[
    \#\left\{1\leq c<D:
    \left|\sum_{n\in L}e\bigl(g(nc)-g(nD)\bigr)\right|
    >\delta\frac ND\right\}\ll_d\delta D.
    \]
    \item There is a nonzero integer $K$ such that
    \[
    |K|\ll_d\delta^{-C_d},\qquad
    \|Kg\|_{C^\infty[N]}\ll_d\delta^{-C_d}.
    \]
\end{enumerate}
\end{lemma}
This lemma, for fixed $D$, bounds the number of integers $c$ with $1\le c<D$ for which the exponential sum over $L$ exceeds $\delta N/D$. This is used to bound the number of bad base points noted in the proof strategy.
The condition on $L$ is taken for various adaptations.
The condition on $D$ is taken to ensure the argument
of \cite[\S3]{GTmob} can be applied.
\begin{proof}
The only essential difference between this proof and the arguments in \cite[\S3]{GTmob} is the case when $D$ is small.
For the reader's convenience, we record the whole proof.

Apply Theorem~\ref{prop:quantitative-weyl-circle} to the polynomials
\(\xi_c(n)=g(nc)-g(nD)=\sum_{j=1}^d\beta_j(c^j-D^j)n^j\) on~$L$, at
parameter \(\delta/(4\pi)\) to absorb \(\|e\|_{\mathrm{Lip}}\).
Write \(\mathcal R\) for the exceptional set in
alternative~(1).  Since \cite{BT} is stated for initial intervals starting at~$1$, translate
\(L=[a,b]\) to \([1,|L|]\) by \(n\mapsto n+a-1\).
The shifted polynomial \(\xi_c(n+a-1)\) has coefficients differing from
those of \(\xi_c\) by factors \(a^{O(d)}\); the bound \(a<\delta^{-C_d-3}\)
keeps the \(C^\infty[N_D]\) size \(\delta^{-O_d(1)}\).

Let $C>0$ be a parameter depending only on $d$.

\smallskip\noindent\textit{Case \(D\le\delta^{-C}\).}
When $D$ is small, we show that either each exponential sum is
at most \(\delta N/D\), or alternative~(2) of the lemma holds.
In other words, the exceptional set in alternative~(1) of the lemma
can be taken empty.

Applying Theorem~\ref{prop:quantitative-weyl-circle} to each \(\xi_c\),
either the exponential sum is \(\le\delta N/D\), or the second
alternative of that theorem produces \(q_c\) with
\(1\le|q_c|\ll_d\delta^{-O_d(1)}\) and
\[
\|q_c\xi_c\|_{C^\infty[N_D]}\ll_d\delta^{-O_d(1)}.
\]
Recall the definition of \(\xi_c\), whose terms contain \(c\) and \(D\).
The second alternative of Theorem~\ref{prop:quantitative-weyl-circle}
therefore reads
\[
N_D^j\bigl\|q_c\beta_j(c^j-D^j)\bigr\|_{\mathbb R/\mathbb Z}
\ll_d\delta^{-O_d(1)}
\qquad(1\le j\le d).
\]
To pass this to alternative~(2) of the lemma, we need the property of $g$.
So we put
\[
K=q_c\prod_{i=1}^d(c^i-D^i),
\qquad
K\beta_j
=\bigl(q_c\beta_j(c^j-D^j)\bigr)
\cdot\prod_{i\neq j}(c^i-D^i).
\]
The second factor is an integer, so \(\|K\beta_j\|_{\mathbb R/\mathbb Z}\) is controlled by
\(\|q_c\beta_j(c^j-D^j)\|_{\mathbb R/\mathbb Z}\).  Changing scale from \(N_D\) to \(N=N_DD\)
costs only powers of \(D\), and since \(D\le\delta^{-C}\) we have
\(|c^j-D^j|\le D^j\le\delta^{-Cj}\).  Thus
\[
0<|K|\ll_d\delta^{-O_d(1)},
\qquad
N^j\|K\beta_j\|_{\mathbb R/\mathbb Z}\ll_d\delta^{-O_d(1)},
\]
which is alternative~(2) of Lemma~\ref{lem:diophantine-nil}.
Since for each $c$, either the exponential sum is at most $\delta N/D$, or alternative~(2) of the lemma holds, we are done.

\smallskip\noindent\textit{Case \(D>\delta^{-C}\).}
If \(|\mathcal R|\le\delta D\) we are done.
Otherwise each \(c\in\mathcal R\) has
\(\bigl|\sum_{n\in L}e(\xi_c(n))\bigr|>\delta N/D\), so
Theorem~\ref{prop:quantitative-weyl-circle} supplies a nonzero integer \(q_c\) with
\(|q_c|\ll_d\delta^{-O_d(1)}\) and \(\|q_c\xi_c\|_{C^\infty[N_D]}\ll_d\delta^{-O_d(1)}\).
The number of possible values of \(q_c\) is \(\ll_d\delta^{-O_d(1)}\).
Since \(|\mathcal R|>\delta D>\delta^{-C+1}\), pigeonholing yields a single nonzero \(q\)
and a set \(E\subseteq\mathcal R\) with
\[
|E|\gg_d\delta^{O_d(1)}\,D
\]
such that \(\|q\xi_c\|_{C^\infty[N_D]}\ll_d\delta^{-O_d(1)}\) for every \(c\in E\).
We choose $C>0$ large enough depending only on $d$ to make sure $E$ is non-empty.
In particular \(E\) has density \(\gg_d\delta^{O_d(1)}\) in \(\{1,\ldots,D-1\}\), and
\[
\bigl\|q\beta_j(c^j-D^j)\bigr\|_{\mathbb R/\mathbb Z}
\ll_d\delta^{-O_d(1)}\,(D/N)^j
\qquad(1\leq j\leq d),\qquad c\in E.
\]

Fix \(j\in\{1,\ldots,d\}\) and set \(\alpha=q\beta_j\).  For every \(c\in E\) we have
\[
\bigl\|\alpha(c^j-D^j)\bigr\|_{\mathbb R/\mathbb Z}
\ll_d\delta^{-O_d(1)}\Bigl(\frac DN\Bigr)^j,
\]
so \(\alpha c^j\) lies in an arc of length
\(\varepsilon_0\ll_d\delta^{-O_d(1)}(D/N)^j\) about \(\alpha D^j\).

To pass from the \(j\)th powers \(c^j\) to a denser set of integers, let
\(t=2^j+1\) and apply Lemma~\ref{lem:gt-waring} to \(E\), with
density \(\eta\gg_d\delta^{O_d(1)}\).  This supplies
\(\gg_d\delta^{O_d(1)}D^j\) integers \(\ell\in[tD^j]\) of the form
\(c_1^j+\cdots+c_t^j\) with each \(c_i\in E\).

For each such \(\ell\), the value \(\alpha\ell\) lies in an arc \(I\) about
\(t\alpha D^j\) of length \(\varepsilon\le t\varepsilon_0\ll_d\delta^{-O_d(1)}(D/N)^j\).
The number of these \(\ell\) is \(\gg_d\delta^{O_d(1)}D^j\), a density
\(\sigma\gg_d\delta^{O_d(1)}\) in \(\{1,\ldots,tD^j\}\).

We apply Lemma~\ref{lem:bt-recurrence}: \(\alpha\ell\in I\) for a density
\(\sigma\) of the integers \(\ell\le tD^j\).  From the bounds on \(\varepsilon\) and \(\sigma\), one has \(\varepsilon/\sigma\ll_d\delta^{-O_d(1)}(D/N)^j\). Thus \(\varepsilon\le\sigma/2\) once \((D/N)^j\ll_d\delta^{O_d(1)}\), which holds for \(N/D>\delta^{-C_d}\) with \(C_d\) large in terms of \(d\).
The lemma produces an integer \(r\ll_d\delta^{-O_d(1)}\) with
\[
D^j\|r\alpha\|_{\mathbb R/\mathbb Z}
\ll_d\delta^{-O_d(1)}\Bigl(\frac DN\Bigr)^j,
\]
hence \(N^j\|rq\beta_j\|_{\mathbb R/\mathbb Z}\ll_d\delta^{-O_d(1)}\).  Doing the same for
each \(j\) produces an integer \(r_j\). Setting \(K=q\prod_{j=1}^d r_j\) gives
\(0<|K|\ll_d\delta^{-O_d(1)}\) and
\(\|Kg\|_{C^\infty[N]}\ll_d\delta^{-O_d(1)}\), which is
alternative~(2) of Lemma~\ref{lem:diophantine-nil}.
\end{proof}

\section{\texorpdfstring{Estimates on $\mathscr{F}_1$}{Estimates on F_1}}
\label{sec:F1}

As in the inclusion--exclusion decomposition, we can decompose the sum
over the union of $\mathscr{F}_i$ into the sum over each of the $\mathscr{F}_i$ and the intersections.
In this section, we deal with the summation over each $\mathscr{F}_i$.
Without loss of generality, we can assume that $i=1$, otherwise we  can apply a permutation 
to change the coordinates.
The contribution of \(\mathscr{F}_1\) is given by the following proposition.
\begin{proposition}\label{prop:F1}
Let $s\geq 2$, and assume that $\mathcal{H}(s-1)$ holds.
For every integer $d\geq 1$, there exist constants $C_d>5$ and $c_0=c_0(d,s)>0$, with $C_d$ depending only on~$d$, such that the following holds.
For every polynomial $g$ of degree~$d$, all
$N$ large enough depending on $d$ and~$s$, all $\delta$ with $N^{-c_0}<\delta<1/8$,
and all integers $A$ in the window
\[
\delta^{-C_d}<A<\delta^{-(C_d+3)},
\]
at least one of the following holds:
\begin{enumerate}
    \item $\big|\Sg{\mathscr{F}_1}-(s-1)!\,N^s\big|\ll_{d,s} \delta^{1/2} N^s (\log N)^{O(s^2)}$.
    \item There exists $q\in\mathbb{Z}\setminus\{0\}$ with $|q|\ll_d\delta^{-O_d(1)}$ such that $\|qg\|_{C^\infty[N]}\ll_d\delta^{-O_d(1)}$.
\end{enumerate}
\end{proposition}

To prove the proposition, split
\[
\Sg{\mathscr{F}_1}
=\Sg{\mathscr{F}_1\setminus\{n_1=m_1\}}
+\Sg{\mathscr{F}_1\cap\{n_1=m_1\}}
=:J_1+J_2.
\]
In the above, $J_1$ is the off-diagonal part, and $J_2$ is the diagonal part.
The off-diagonal part \(J_1\) has no main term.
The diagonal part \(J_2\) contains main term from
\(\mathcal{H}(s-1)\).
In Section \ref{sec:off-diagonal}, we deal with the off-diagonal part $J_1$, and in Section \ref{sec:diagonal}, we deal with the diagonal part $J_2$.

Recall that the parameter $A$ appears in the definition of $\mathscr{F}_1$, and is restricted to $\delta^{-C_d}<A<\delta^{-(C_d+3)}$. The lower bound matches the hypothesis $N/D>\delta^{-C_d}$ of Lemma~\ref{lem:diophantine-nil}. It is checked below, so the lemma can be applied.

\subsection{\texorpdfstring{Off-diagonal part $J_1$}{Off-diagonal part J_1}}
\label{sec:off-diagonal}

The main result of this subsection is Lemma~\ref{lem:J1}.
Write $\mathscr{F}_1^\circ:=\mathscr{F}_1\setminus\{n_1=m_1\}$ for the off-diagonal part of~$\mathscr{F}_1$, so that $J_1=\Sg{\mathscr{F}_1^\circ}$.

To prove Lemma~\ref{lem:J1}, we first need a parameterisation lemma for the off-diagonal part, which is slightly
different from \cite[Lemma~3.4]{KX} (the case $s=2$) and the factorisations in
\cite[Lemma~2.3]{BNR} and \cite[\S8]{VW}. 
\begin{lemma}[Matrix parameterisation]\label{lem:param-general}
Let $H\geq 1$ and, with \(V\) as above, set
\[
V_H^\circ=\bigl\{(\vec{n}_s,\vec{m}_s)\in V:\ \gcd(n_1,m_1)\geq H,\ n_1\neq m_1\bigr\}.
\]
There is an injection $\phi\colon V_H^\circ\to M_{s\times s}(\mathbb{Z}_{>0})$ sending $(\vec{n}_s,\vec{m}_s)$ to a matrix
\[
\begin{pmatrix}
a_{11} & a_{12} & \cdots & a_{1s} \\
a_{21} & a_{22} & \cdots & a_{2s} \\
\vdots & \vdots & \ddots & \vdots \\
a_{s1} & a_{s2} & \cdots & a_{ss}
\end{pmatrix}
\]
with $n_i=\prod_j a_{ij}$ the product of row $i$ and $m_i=\prod_k a_{ki}$ the product of column $i$.
Write $\pi_1$ for the projection onto the off-diagonal entries,
\[
\pi_1\bigl((a_{ij})_{1\leq i,j\leq s}\bigr)
=(a_{ij})_{i\neq j},
\]
forgetting the diagonal.
For each $y=(a_{ij})_{i\neq j}\in\pi_1\bigl(\phi(V_H^\circ)\bigr)$, writing
\[
A_i=\prod_{j\neq i}a_{ij},\qquad B_i=\prod_{k\neq i}a_{ki},\qquad M_i=\max\{A_i,B_i\},
\]
one has $A_1\neq B_1$, and
\begin{equation}\label{eq:fibre-general}
\pi_1^{-1}(y)\cap\phi(V_H^\circ)
=\left\{
(a_{ij}):
(a_{ij})_{i\neq j}=y,\quad
a_{11}\in[H,N/M_1],\quad
a_{ii}\in[1,N/M_i]\ (2\leq i\leq s)
\right\}.
\end{equation}
\end{lemma}
We take a moment to explain the intuition behind the parameterisation lemma.
At \(s=2\), \cite[Lemma~3.4]{KX} parameterises the solutions of \(m_1m_2=n_1n_2\) by
\[
m_1=ga,\qquad m_2=hb,\qquad n_1=gb,\qquad n_2=ha,
\qquad \gcd(a,b)=1,
\]
where \(g=\gcd(m_1,n_1)\) and \(h=\gcd(m_2,n_2)\).

To estimate the exponential sum over \(\mathscr{F}_1^\circ\), it is important to find arithmetic progressions in the parameterisation. At \(s=2\), fix the coprime pair \((a,b)\). Both \(g\) and \(h\) stay free. Then we may first sum over \(g\) and \(h\). As \(g\) runs, \(n_1=gb\) and \(m_1=ga\) are arithmetic progressions. As \(h\) runs, \(n_2=ha\) and \(m_2=hb\) are arithmetic progressions. This is the structure used in \cite[Lemma~3.4]{KX} for the computation of the 4th moment. The lemma above keeps it and extends the parameterisation to general \(s\).
For general $s$, the lemma describes the fibres of $\pi_1$ on the image of $\phi$.  In particular, as $a_{11}$ varies,
$n_1=a_{11}A_1$ and
$m_1=a_{11}B_1$ are arithmetic progressions.
This is the most important property used in what follows, and the main addition
compared with \cite[Lemma~2.3]{BNR} and \cite[\S8]{VW}.
\begin{proof}
We first construct $\phi$ and check injectivity, then verify the fibre description.

\smallskip\noindent\textbf{Construction.}
Set $a_{ii}:=\gcd(n_i,m_i)$ for every~$i$. Write $u_i:=n_i/a_{ii}$ and
$v_i:=m_i/a_{ii}$, so that $\gcd(u_i,v_i)=1$ for each~$i$ and
$\prod_i u_i=\prod_i v_i$.
The off-diagonal entries $a_{ij}$ with $i\neq j$ remain to be chosen so that
\[
u_i=\prod_{j\neq i}a_{ij},\qquad
v_i=\prod_{k\neq i}a_{ki}.
\]
For the off-diagonal entries, the way we choose them does not matter so much as long as the procedure makes the parameterisation unique.
One may instead fill them by \cite[Lemma~2.3]{BNR} or by the construction of \cite[\S8]{VW}.
For the reader's convenience we record the following construction.
We proceed by induction on~$s$.
Extract the off-diagonal entries of the first row by successive gcds
\[
a_{1j}=\gcd\Bigl(u_1\Big/\prod_{2\leq t<j}a_{1t},\,v_j\Bigr)
\qquad(j\geq 2).
\]
We claim that $u_1=\prod_{j=2}^s a_{1j}$.
In fact, from $\prod_i u_i=\prod_i v_i$ we have $u_1\mid\prod_i v_i$.
Since also $\gcd(u_1,v_1)=1$, we get $u_1\mid\prod_{j=2}^s v_j$.
Write $L$ for the leftover $u_1\big/\prod_{j=2}^s a_{1j}$.
Then $L$ divides $\prod_{j=2}^s v_j$, but is coprime to each $v_j$ with $j\geq 2$, so $L=1$.
Similarly, extract the off-diagonal entries of the first column by successive gcds
\[
a_{i1}=\gcd\Bigl(u_i,\,v_1\Big/\prod_{2\leq t<i}a_{t1}\Bigr)
\qquad(i\geq 2).
\]
The same argument gives $v_1=\prod_{i=2}^s a_{i1}$.

When $s=2$, the off-diagonal consists exactly of $a_{12}$ and $a_{21}$, so $u_1=a_{12}$ and $v_1=a_{21}$. The identity $u_1 u_2=v_1 v_2$ together with $\gcd(u_2,v_2)=1$ then gives $u_2=a_{21}$ and $v_2=a_{12}$.

Now suppose $s\geq 3$.
Write $u_i':=u_i/a_{i1}$ and $v_i':=v_i/a_{1i}$ for $i=2,\ldots,s$.
These are positive integers with $\gcd(u_i',v_i')=1$. Cancelling
$u_1=\prod_{j\geq 2}a_{1j}$ and $v_1=\prod_{i\geq 2}a_{i1}$ from the
product equation gives
\[
u_2'\cdots u_s'=v_2'\cdots v_s'.
\]
Thus the same construction applies to $u_i'$, $v_i'$,
and uniquely determines the remaining $(s-1)\times(s-1)$ block. That is,
\[
u_i'=\prod_{\substack{j=2\\ j\neq i}}^s a_{ij},\qquad
v_i'=\prod_{\substack{k=2\\ k\neq i}}^s a_{ki}
\qquad(i=2,\ldots,s).
\]
and therefore
\[
u_i=a_{i1}u_i'=\prod_{j\neq i}a_{ij},\qquad
v_i=a_{1i}v_i'=\prod_{k\neq i}a_{ki}.
\]

Together with $a_{ii}=\gcd(n_i,m_i)$, this gives $n_i=\prod_j a_{ij}$ and $m_i=\prod_k a_{ki}$.
Note that in this construction, $\gcd(A_i,B_i)=1$ for $i=1,\ldots,s$.
Also note that
this construction procedure makes the parameterisation unique.

\smallskip\noindent\textbf{Injectivity.}
The map is injective because the pair is recovered as the row and column products $n_i=\prod_j a_{ij}$ and $m_i=\prod_k a_{ki}$.

\smallskip\noindent\textbf{Fibre.}
Fix $y=(a_{ij})_{i\neq j}\in\pi_1\bigl(\phi(V_H^\circ)\bigr)$, and write $A_i$, $B_i$ and $M_i$ as in the statement.
Since $y$ arises from a matrix in $\phi(V_H^\circ)$, one has  $A_1\neq B_1$ and $\gcd(A_1,B_1)=1$ by the construction.

We check the two inclusions in~\eqref{eq:fibre-general}.
We first show that $\pi_1^{-1}(y)\cap\phi(V_H^\circ)$ is contained in the right-hand side.
If $(a_{ij})\in\pi_1^{-1}(y)\cap\phi(V_H^\circ)$, then $(a_{ij})=\phi(\vec{n}_s,\vec{m}_s)$ for some $(\vec{n}_s,\vec{m}_s)\in V_H^\circ$, and
\[
n_i=a_{ii}A_i,\qquad m_i=a_{ii}B_i
\qquad(1\leq i\leq s).
\]
Membership in $V_H^\circ$ and $\gcd(A_1,B_1)=1$ give $a_{11}=\gcd(n_1,m_1)\geq H$.
The bounds $n_i,m_i\leq N$ give $a_{ii}\leq N/M_i$ for each $i$.
Thus $(a_{ij})$ lies in the right-hand side of~\eqref{eq:fibre-general}.

We next show that the right-hand side of~\eqref{eq:fibre-general} is contained in $\pi_1^{-1}(y)\cap\phi(V_H^\circ)$.
Let $a_{11}\in[H,N/M_1]$ and $a_{ii}\in[1,N/M_i]$ for $2\leq i\leq s$, and let $(a_{ij})$ be the matrix with this diagonal and with off-diagonal entries given by $y$.
Reconstruct
\[
n_i=a_{ii}A_i,\qquad m_i=a_{ii}B_i
\qquad(1\leq i\leq s).
\]
We claim that $(\vec n_s,\vec m_s)$ lies in $V_H^\circ$ and that $\phi(\vec n_s,\vec m_s)=(a_{ij})$.

We first check membership in $V$.
The bounds $a_{ii}\leq N/M_i$ and $A_i,B_i\leq M_i$ give $n_i,m_i\leq N$.
The coordinates are positive integers.
The product identity $n_1\cdots n_s=m_1\cdots m_s$ holds because $\prod_i A_i$ and $\prod_i B_i$ are both the product of the off-diagonal entries of $y$.

We next check the gcd condition.
Since $\gcd(A_1,B_1)=1$, we have $\gcd(n_1,m_1)=a_{11}\geq H$.
Finally, $A_1\neq B_1$ and $a_{11}\geq 1$ give $n_1=a_{11}A_1\neq a_{11}B_1=m_1$.
Thus the reconstructed tuple lies in $V_H^\circ$.

It remains to check that $\phi(\vec n_s,\vec m_s)=(a_{ij})$.
Since $\gcd(A_i,B_i)=1$ for each $i$, the diagonal of $\phi(\vec n_s,\vec m_s)$ is recovered as $a_{ii}=\gcd(n_i,m_i)$.
Uniqueness of the parameterisation then forces the off-diagonal to be exactly $y$.
\end{proof}

\begin{remark}\label{rem:param-remark}
The same parameterisation $\phi$ applies to any finite set of positive solutions of
$n_1\cdots n_s=m_1\cdots m_s$
(in particular to the scaled images in Piece~III of Lemma~\ref{lem:I1}).
\end{remark}

We now deal with $J_1$. The following lemma is the goal of this section.

\begin{lemma}\label{lem:J1}
Let the parameters be as in Proposition~\ref{prop:F1}.
At least one of the following holds:
\begin{enumerate}
    \item $|J_1|\ll_{d,s} \delta^{1/2} N^s(\log N)^{O(s^2)}$.
    \item There exists $q\in\mathbb{Z}\setminus\{0\}$ with $|q|\ll_d\delta^{-O_d(1)}$ such that $\|qg\|_{C^\infty[N]}\ll_d\delta^{-O_d(1)}$.
\end{enumerate}
\end{lemma}

We now explain the proof strategy for Lemma~\ref{lem:J1}.

Apply Lemma~\ref{lem:param-general} with $H=A$ to
$\mathscr{F}_1^\circ=V_A^\circ$, and set
\[
\mathscr{F}_1'
:=\pi_1\bigl(\phi(\mathscr{F}_1^\circ)\bigr)
\]
for the resulting set of off-diagonal base points.
In Figure~\ref{fig:roadmap}, \(\phi\) is the parameterisation and \(\pi_1\) is the projection; \(\pi_1\) forgets the diagonal.
Each point of \(\mathscr{F}_1'\) is a base point.  The fibre over it is given by~\eqref{eq:fibre-general}.  Figure~\ref{fig:pi1} draws each fibre as a vertical interval.

\begin{figure}[ht]
\centering
\begin{tikzpicture}[
  font=\small,
  >=Stealth,
  arr/.style={->, thick},
]
  \node[draw=figink, rounded corners=2pt, fill=figplane,
        minimum width=1.8cm, minimum height=1.8cm]
    (F1) at (0,2.55) {};
  \fill[plum!34] ($(F1.center)+(0,0.16)$) ellipse (0.34 and 0.46);
  \draw[figink!75, line width=0.55pt] ($(F1.center)+(0,0.16)$) ellipse (0.34 and 0.46);
  \node[font=\scriptsize, anchor=north, inner sep=0pt]
    at ($(F1.center)+(0,-0.46)$) {$\mathscr{F}_1$};

  \node[inner sep=1pt, font=\scriptsize] (Phi) at (5.1,2.55) {%
    $\begin{pmatrix}
    a_{11} & a_{12} & \cdots & a_{1s} \\
    a_{21} & a_{22} & \cdots & a_{2s} \\
    \vdots & \vdots & \ddots & \vdots \\
    a_{s1} & a_{s2} & \cdots & a_{ss}
    \end{pmatrix}$};
  \node[above=2pt] at (Phi.north) {$\phi(\mathscr{F}_1)$};

  \draw[arr] (F1.east) -- node[above]{$\phi$} (Phi.west);

  \coordinate (O) at (4.95,-1.35);
  \coordinate (U) at (1.35,0.38);
  \coordinate (V) at (-1.05,0.52);
  \def\ah{0.42}
  \fill[figplane] (O) -- ++(U) -- ++(V) -- ++($-1*(U)$) -- cycle;
  \draw[figink] (O) -- ++(U) -- ++(V) -- ++($-1*(U)$) -- cycle;
  \node[below=1pt] at ($(O)+0.55*(U)$) {$\mathscr{F}_1'$};

  \draw[arr] (Phi.south) -- node[right]{$\pi_1$} ++(0,-0.75);

  % back to front
  \foreach \s/\t/\h in {
    0.20/0.80/0.58,
    0.70/0.74/0.92,
    0.42/0.50/1.28,
    0.82/0.36/0.68,
    0.26/0.20/1.02,
    0.60/0.16/1.18%
  } {
    \coordinate (P) at ($(O)+\s*(U)+\t*(V)$);
    \coordinate (Bot) at ($(P)+(0,\ah)$);
    \coordinate (Top) at ($(Bot)+(0,\h)$);
    \draw[figink!40, dotted] (P) -- (Bot);
    \draw[figblue, line width=1.35pt] (Bot) -- (Top);
    \fill[figblue] (Bot) circle (1.15pt);
    \fill[figblue] (Top) circle (0.7pt);
  }
\end{tikzpicture}
\caption{Each point of \(\mathscr{F}_1'\) is a base point.  The vertical
intervals are the fibres of~\eqref{eq:fibre-general}.
}
\label{fig:pi1}
\end{figure}
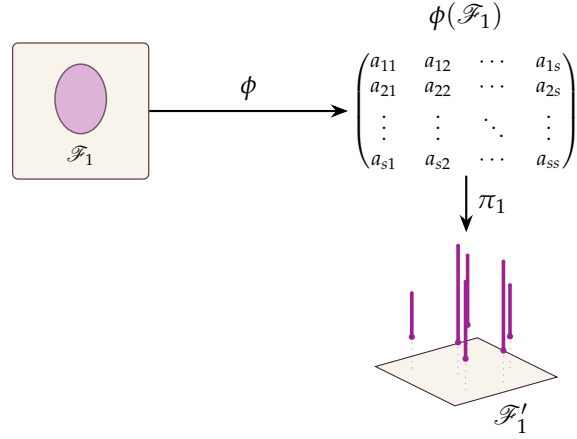

A first attempt uses only Lemma~\ref{lem:param-general} and $|e|=1$.
Each fibre has length at most \(\prod_{i=1}^s N/M_i\), and summing over the
base points in \(\mathscr{F}_1'\) gives
\[
|J_1|\ll N^s(\log N)^{s(s-1)}\ll_s N^s(\log N)^{O(s^2)}.
\]
This is even larger than the size of the main term \(s!\,N^s\), so this is not enough.

The proof uses Lemma~\ref{lem:diophantine-nil} to get a $\delta$ saving.
We express \(J_1\) by summing first over the fibres with fixed base point,
then summing over all base points. The base points are then split into good and bad according to the size of the fibre sum.
For the bad base points, we show that they take up only a small proportion by using Lemma~\ref{lem:diophantine-nil}.

\begin{proof}
Assume alternative~(2) of the present lemma fails.
We start from $J_1=\Sg{\mathscr{F}_1^\circ}$.
Lemma~\ref{lem:param-general} with $H=A$ sends each $(\vec{n}_s,\vec{m}_s)\in\mathscr{F}_1^\circ$ to a matrix $(a_{ij})$ with $n_i=\prod_j a_{ij}$ and $m_i=\prod_k a_{ki}$.
Recall that
\[
\mathscr{F}_1'
=\pi_1\bigl(\phi(\mathscr{F}_1^\circ)\bigr).
\]
For each $(a_{ij})_{i\neq j}\in\mathscr{F}_1'$, Lemma~\ref{lem:param-general} gives the fibre~\eqref{eq:fibre-general} with $H=A$, and $n_i=a_{ii}A_i$, $m_i=a_{ii}B_i$ for $1\leq i\leq s$.
Substituting into $J_1$ gives
\[
J_1=\sum_{\mathscr{F}_1'}\sum_{\substack{a_{11}\in[A,N/M_1]\\ a_{ii}\in[1,N/M_i]\ (i\geq 2)}}\prod_{i=1}^s e\bigl(g(a_{ii}A_i)-g(a_{ii}B_i)\bigr).
\]
The outer sum runs over the base points, and the inner sum is the exponential sum over one fibre, as in Figure~\ref{fig:pi1}.
Since $|e|=1$, the sums over $a_{ii}$ for $i\geq 2$ are at most $N/M_i$, and we obtain
\begin{equation}\label{eq:J1-bound}
|J_1|\leq \sum_{\mathscr{F}_1'}\Big|\sum_{A\leq a_{11}\leq N/M_1} e\big(g(a_{11}A_1)-g(a_{11}B_1)\big)\Big|
\prod_{i=2}^s\frac{N}{M_i}.
\end{equation}

By symmetry between $A_1$ and $B_1$, sum separately on $\{A_1<B_1\}$ and $\{A_1>B_1\}$; we treat $\{A_1<B_1\}$ (so $M_1=B_1$) and the other side is identical.

For a base point $(a_{ij})_{i\neq j}\in\mathscr{F}_1'$, write
\[
S=\sum_{A\leq a_{11}\leq N/B_1} e\big(g(a_{11}A_1)-g(a_{11}B_1)\big).
\]
We call the right endpoint $N/B_1$ the height of the fibre. The lower bound $A>\delta^{-C_d}$ is imposed so that this height exceeds $\delta^{-C_d}$, as required by Lemma~\ref{lem:diophantine-nil}. This is checked in the proof of the following claim.
The strategy is to apply Lemma~\ref{lem:diophantine-nil} to $S$ and split the fibres of Figure~\ref{fig:pi1} into good and bad.
On a good fibre, $|S|\ll\delta N/B_1$.
The bad fibres are those for which the exponential sum is large.
This is recorded as the following claim.

\begin{clm}\label{clm:J1-fibre}
Let \((a_{ij})_{i\neq j}\in\mathscr{F}_1'\) be a base point with \(A_1<B_1\). For \(S\) as above, there exists a set \(\mathcal{R}_{B_1}\subset\{1,\ldots,B_1-1\}\), depending only on \(B_1\), with \(|\mathcal{R}_{B_1}|\ll_d\delta B_1\) such that
\[
|S|\ \leq\ 2\delta\frac{N}{B_1}+\mathbf{1}_{A_1\in\mathcal{R}_{B_1}}\cdot\frac{N}{B_1}.
\]
\end{clm}

\noindent\textit{Proof of the claim.}
We plan to apply Lemma~\ref{lem:diophantine-nil}.
But the input of this lemma requires the length of the interval to be comparable to the height.
Since there are long fibres and short fibres as shown in Figure~\ref{fig:J1-longshort},
one may worry that if the fibre is too short, then we are not able to apply Lemma~\ref{lem:diophantine-nil}.

But the observation is that we can still apply Lemma~\ref{lem:diophantine-nil}, because the height of the fibre is large enough.
We divide into two cases.

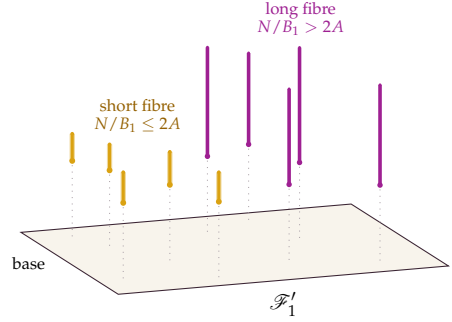
\begin{figure}[ht]
\centering
\begin{tikzpicture}[
  font=\small,
  scale=0.78,
  every node/.append style={transform shape},
]
  \coordinate (O) at (0,0);
  \coordinate (U) at (5.6,0.48);
  \coordinate (V) at (-1.65,1.05);
  \def\ah{1.28}

  \fill[figplane] (O) -- ++(U) -- ++(V) -- ++($-1*(U)$) -- cycle;
  \draw[figink] (O) -- ++(U) -- ++(V) -- ++($-1*(U)$) -- cycle;
  \node[below=2pt] at ($(O)+0.50*(U)$) {$\mathscr{F}_1'$};

  \foreach \s/\t/\len in {
    0.08/0.22/0.52,
    0.18/0.70/0.44,
    0.28/0.42/0.56,
    0.12/0.88/0.46,
    0.34/0.12/0.50%
  } {
    \coordinate (P) at ($(O)+\s*(U)+\t*(V)$);
    \coordinate (Bot) at ($(P)+(0,\ah)$);
    \coordinate (Top) at ($(Bot)+(0,\len)$);
    \draw[figink!40, dotted] (P) -- (Bot);
    \draw[goldpale, line width=2pt] (Bot) -- (Top);
    \draw[gold, line width=1.15pt] (Bot) -- (Top);
    \fill[gold] (Bot) circle (1.4pt);
    \fill[gold] (Top) circle (0.9pt);
  }

  \foreach \s/\t/\len in {
    0.50/0.78/1.85,
    0.60/0.28/1.62,
    0.72/0.58/1.95,
    0.84/0.16/1.70,
    0.66/0.90/1.55%
  } {
    \coordinate (P) at ($(O)+\s*(U)+\t*(V)$);
    \coordinate (Bot) at ($(P)+(0,\ah)$);
    \coordinate (Top) at ($(Bot)+(0,\len)$);
    \draw[figink!40, dotted] (P) -- (Bot);
    \draw[figblue, line width=1.35pt] (Bot) -- (Top);
    \fill[figblue] (Bot) circle (1.3pt);
    \fill[figblue] (Top) circle (0.7pt);
  }

  \node[font=\scriptsize, align=center, text=goldink, anchor=south]
    at ($(O)+0.20*(U)+0.48*(V)+(0,\ah)+(0,0.74)$)
    {short fibre\\[-1pt]$N/B_1\le 2A$};
  \node[font=\scriptsize, align=center, text=figblue, anchor=south]
    at ($(O)+0.70*(U)+0.50*(V)+(0,\ah)+(0,2.12)$)
    {long fibre\\[-1pt]$N/B_1>2A$};
  \node[font=\scriptsize, left=8pt] at ($(O)+0.50*(V)$) {base};
\end{tikzpicture}
\caption{The fibre interval \(L=[A,N/B_1]\) is drawn floating above \(\mathscr{F}_1'\).
Gold: short fibres \(A\le N/B_1\le 2A\).
Purple: long fibres \(N/B_1>2A\).}
\label{fig:J1-longshort}
\end{figure}

\smallskip\noindent\emph{Long fibres} ($N/B_1>2A$).
In this case the range \(L:=[A,N/B_1]\) is an interval of length \(\asymp N/B_1\) with left endpoint \(A<\delta^{-(C_d+3)}\), so alternative~(1) of Lemma~\ref{lem:diophantine-nil} applies directly. There is \(\mathcal{R}_{B_1}\subset\{1,\ldots,B_1-1\}\) with \(|\mathcal{R}_{B_1}|\ll_d\delta B_1\) such that \(|S|\leq\delta N/B_1\) for all \(A_1\notin\mathcal{R}_{B_1}\).
Without the length comparison, \(L\) would be an illegal input for that lemma.

\smallskip\noindent\emph{Short fibres.}
Here $A\le N/B_1\le 2A$, so $N/B_1\asymp A$.
We write $S$ as a difference of two initial segments,
\[
S=\sum_{1\leq a_{11}\leq N/B_1} e\big(g(a_{11}A_1)-g(a_{11}B_1)\big)
-\sum_{1\leq a_{11}\leq A-1} e\big(g(a_{11}A_1)-g(a_{11}B_1)\big).
\]
Both \([1,N/B_1]\) and \([1,A-1]\) then have length comparable to \(N/B_1\), and both start at \(1\), so Lemma~\ref{lem:diophantine-nil} applies to each.
Let \(\mathcal{R}_{B_1}\) be the union of the two exceptional sets. One still has \(|\mathcal{R}_{B_1}|\ll_d\delta B_1\).
For \(A_1\notin\mathcal{R}_{B_1}\), each of the two sums is at most \(\delta N/B_1\), and therefore \(|S|\le 2\delta N/B_1\).
The factor $2$ is the only extra cost of writing the short fibre as a difference of two longer sums.

In either case one has $|S|\le 2\delta N/B_1$ off a set $\mathcal{R}_{B_1}$ of size $\ll_d\delta B_1$.  On that set the trivial bound $|S|\le N/B_1$ gives the claimed inequality.

This proves the claim.

\smallskip
Recall that $S$ is obtained by fixing the base point $(a_{ij})_{i\neq j}$ and summing over $a_{11}$.
The claim therefore groups the base points, or equivalently the fibres, into good and bad: the good fibres contribute the first term, with a $2\delta$ saving, while the bad fibres---those $(a_{ij})_{i\neq j}$ with $A_1\in\mathcal{R}_{B_1}$---contribute the second term.
By the claim one has $|\mathcal{R}_{B_1}|\ll_d\delta B_1$, so the number of bad fibres is small, which will be helpful when summing over the base points.

Next we sum over the base points, that is, $\mathscr{F}_1'$.
Since $M_1=B_1$ in this branch, $\dfrac{N}{B_1}\prod_{i=2}^s\dfrac{N}{M_i}=\prod_{i=1}^s\dfrac{N}{M_i}$.
Write
\[
\Sigma_{\mathrm I}:=\sum_{\mathscr{F}_1'\cap\{A_1<B_1\}}\ \prod_{i=1}^s\frac{N}{M_i},\qquad
\Sigma_{\mathrm{II}}:=\sum_{\mathscr{F}_1'\cap\{A_1<B_1,\,A_1\in\mathcal{R}_{B_1}\}}\ \prod_{i=1}^s\frac{N}{M_i}.
\]
Inserting the claim, the contribution of $\{A_1<B_1\}$ to the right-hand side of~\eqref{eq:J1-bound} satisfies
\[
\sum_{\mathscr{F}_1'\cap\{A_1<B_1\}}\ |S|\prod_{i=2}^s\frac{N}{M_i}
\leq 2\delta\Sigma_{\mathrm I}+\Sigma_{\mathrm{II}}.
\]

\smallskip\noindent\emph{The term $\Sigma_{\mathrm I}$.}
Since $M_i\geq A_i:=\prod_{j\neq i}a_{ij}$ for every $i=1,\ldots,s$, we may replace each $M_i$ by $A_i$ and drop the restriction $A_1<B_1$.
Summing row by row,
\[
\Sigma_{\mathrm I}\ \leq\ \sum_{\mathscr{F}_1'}\prod_{i=1}^s\frac{N}{A_i}
\leq\prod_{i=1}^s\Big(N\sum_{A_i\leq N}\frac{\tau_{s-1}(A_i)}{A_i}\Big).
\]
By Lemma~\ref{lem:divisor}, $\sum_{A_i\leq N}\tau_{s-1}(A_i)/A_i\leq(2\log N)^{s-1}$ for each of the $s$ rows, so
\[
\Sigma_{\mathrm I}\ll N^s(\log N)^{s(s-1)}\ll_s N^s(\log N)^{O(s^2)}.
\]

\smallskip\noindent\emph{The term $\Sigma_{\mathrm{II}}$.}
We are still in the branch $A_1<B_1$, so $M_1=B_1$.
For $i\ge 2$, we have $M_i\ge A_i$, and we use this trivial bound  $\prod_{i=2}^sN/M_i\le \prod_{i=2}^s N/A_i$.
So $\Sigma_{\mathrm{II}}$ is less than
\begin{equation}\label{eq:Sigma-II}
\sum_{\mathscr{F}_1'\cap\{A_1<B_1,\,A_1\in\mathcal{R}_{B_1}\}}\ \frac{N}{B_1}\prod_{i=2}^s\frac{N}{A_i}.
\end{equation}
The following picture, at $s=5$, records which entries of the matrix in Lemma~\ref{lem:param-general} appear in this denominator; the $A_i$ and $B_i$ are the off-diagonal row and column products.
\begin{center}
\begin{tikzpicture}[
  font=\tiny,
  cell/.style={draw, minimum size=4.1mm, inner sep=0pt},
]
  % white: absent; mid: once (rows 2..s); dark: twice (B_1 and A_i)
  \foreach \i/\ii in {1/1, 2/2, 3/3, 4/4, 5/s} {
    \foreach \j/\jj in {1/1, 2/2, 3/3, 4/4, 5/s} {
      \ifnum\i=\j
        \def\lab{\cdot}
        \def\cellfill{white}
      \else
        \def\lab{a_{\ii\jj}}
        \ifnum\i=1
          \def\cellfill{white}
        \else
          \ifnum\j=1
            \def\cellfill{figcell}
          \else
            \def\cellfill{figcellpale}
          \fi
        \fi
      \fi
      \node[cell, draw=figink!45, fill=\cellfill] (C\i\j) at ({(\j-1)*0.46},{(1-\i)*0.46}) {$\lab$};
    }
  }
  \foreach \i/\lab in {1/{A_1}, 2/{A_2}, 3/{A_3}, 4/{A_4}, 5/{A_s}} {
    \node[left=1.5pt] at (C\i1.west) {\(\lab\)};
  }
  \foreach \j/\lab in {1/{B_1}, 2/{B_2}, 3/{B_3}, 4/{B_4}, 5/{B_s}} {
    \node[below=1.5pt] at (C5\j.south) {\(\lab\)};
  }
\end{tikzpicture}
\end{center}
The first row is white: those entries do not appear in the denominator.
The darker cells appear twice in the denominator, and the paler cells appear once.
Since the first row does not appear in the denominator, we may first fix rows $2,\ldots,s$ and sum over the first row, without worrying about the weight changing in \eqref{eq:Sigma-II}.

Fix rows $2,\ldots,s$. Then $B_1$ and $A_2,\ldots,A_s$ are constant, so the weight in~\eqref{eq:Sigma-II} stays constant, and we sum first over the first row.
For each fixed $A_1$, the number of tuples $(a_{12},\ldots,a_{1s})$ with $A_1=\prod_{j=2}^s a_{1j}$ is at most $\tau_{s-1}(A_1)$.
Thus the inner sum over the first row contributes a factor at most
\[
\sum_{\substack{A_1<B_1\\ A_1\in\mathcal{R}_{B_1}}}\tau_{s-1}(A_1)
\leq\sum_{A_1\in\mathcal{R}_{B_1}}\tau_{s-1}(A_1).
\]

By Cauchy--Schwarz and the second bound of Lemma~\ref{lem:divisor}, we have
\begin{align*}
\sum_{A_1\in\mathcal{R}_{B_1}}\tau_{s-1}(A_1)
&\leq |\mathcal{R}_{B_1}|^{1/2}\Bigl(\sum_{n\leq B_1}\tau_{s-1}(n)^2\Bigr)^{1/2}\\
&\ll_{d,s}\delta^{1/2}B_1(\log N)^{O(s^2)}.
\end{align*}

We next sum over rows $2,\dots, s$.
Multiplying the above count by $\dfrac{N}{B_1}\prod_{i=2}^s\dfrac{N}{A_i}$ as in~\eqref{eq:Sigma-II} and summing over rows $2,\ldots,s$,
the factors of $B_1$ cancel and we obtain
\[
\Sigma_{\mathrm{II}}\ \ll_{d,s}\ \delta^{1/2}N(\log N)^{O(s^2)}\sum_{\text{rows }2,\ldots,s}\prod_{i=2}^s\frac{N}{A_i}.
\]
The remaining sum is a product of row-sums exactly as for $\Sigma_{\mathrm I}$, so
\[
\Sigma_{\mathrm{II}}\ll_d \delta^{1/2} N^s(\log N)^{O(s^2)}.
\]
Combining, we have $2\delta\Sigma_{\mathrm I}+\Sigma_{\mathrm{II}}\ll_{d,s}\delta^{1/2} N^s(\log N)^{O(s^2)}$. The same bound holds on $\{A_1>B_1\}$ by symmetry.
Thus, unless the Diophantine alternative of Lemma~\ref{lem:diophantine-nil} occurs, $|J_1|\ll_{d,s} \delta^{1/2} N^s(\log N)^{O(s^2)}$.
If it does occur, the integer $K$ produced there is the $q$ of alternative~(2).
\end{proof}

\subsection{\texorpdfstring{Diagonal part $J_2$}{Diagonal part J_2}}
\label{sec:diagonal}

We now deal with $J_2$. The following lemma is the goal of this section. The argument uses the inductive hypothesis $\mathcal{H}(s-1)$.
\begin{lemma}\label{lem:J2}
Assume that $\mathcal{H}(s-1)$ holds.
There exists $c_0=c_0(d,s)>0$ such that, for the remaining parameters as in Proposition~\ref{prop:F1}, with this choice of $c_0$, at least one of the following alternatives occurs:
\begin{enumerate}
    \item $\big|J_2-(s-1)!\,N^s\big|\ll_{d,s}\delta^{1/2} N^s(\log N)^{O(s^2)}$.
    \item There exists $q\in\mathbb{Z}\setminus\{0\}$ with $|q|\ll_d\delta^{-O_d(1)}$ such that $\|qg\|_{C^\infty[N]}\ll_d\delta^{-O_d(1)}$.
\end{enumerate}
\end{lemma}
\begin{proof}
On $\mathscr{F}_1\cap\{n_1=m_1\}$ the constraint $\gcd(n_1,m_1)\geq A$ defining $\mathscr{F}_1$ reduces to $n_1\geq A$, and the diagonal factor $e(g(n_1))\overline{e(g(n_1))}$ equals $1$.
Summing freely over $n_1\in[A,N]$ and over the remaining variables, which satisfy $n_2\cdots n_s=m_2\cdots m_s$,
\[
J_2=(N-A+1)\,U_{s-1}(N).
\]
Writing $U_{s-1}(N)=N^{s-1}\mathcal{M}_{s-1}(N)$ and expanding $(N-A+1)N^{s-1}=N^s+O(AN^{s-1})$,
\begin{equation}\label{eq:J2-expand}
J_2=N^s\,\mathcal{M}_{s-1}(N)+O\!\big(AN^{s-1}\mathcal{M}_{s-1}(N)\big).
\end{equation}

If $s=2$, $\mathcal{H}(1)$ gives $\mathcal{M}_1(N)=1$ exactly, so \eqref{eq:J2-expand} reads $J_2=N^2+O(AN)$.
The upper bound $A<\delta^{-(C_d+3)}$ together with $\delta>N^{-c_0}$ (with $c_0$ small enough) yields
$A\ll_{d}\delta N$, hence $O(AN)\ll_{d}\delta N^2$, which is alternative~(1).

If $s\geq 3$, apply the inductive hypothesis $\mathcal{H}(s-1)$ to $\mathcal{M}_{s-1}(N)$.
If it produces alternative~(2), the same $q$ serves here and we are done.
Otherwise
\[
\big|\mathcal{M}_{s-1}(N)-(s-1)!\big|\ll_{d,s}\delta(\log N)^{O(s^2)},
\]
and so $\mathcal{M}_{s-1}(N)\ll_{d,s}(\log N)^{O(s^2)}$.
 The same $A\ll_d \delta N$ also bounds the error term.
Hence \eqref{eq:J2-expand} gives
\[
\bigl|J_2-(s-1)!\,N^s\bigr|\ll_{d,s}\delta^{1/2} N^s(\log N)^{O(s^2)},
\]
which is alternative~(1).
\end{proof}

Combining Lemmas~\ref{lem:J1} and~\ref{lem:J2} proves Proposition~\ref{prop:F1}.

\begin{remark}[Symmetry among $\mathscr{F}_i$]\label{rem:Fi-symmetry}
For $i\in\{1,\ldots,s\}$, relabel $(n_1,m_i)\mapsto(m_i,n_1)$ and permute the remaining indices so that $\mathscr{F}_i$ becomes $\mathscr{F}_1$ with $\gcd(n_1,m_1)\geq A$ unchanged.  Proposition~\ref{prop:F1} therefore applies identically to $\sum_{\mathscr{F}_i}$, with the same main term $(s-1)!\,N^s$ and the same error bound.
\end{remark}

\section{Estimates on intersections}
\label{sec:intersection}

For $2\leq\kappa\leq s$ and distinct indices $\{i_1,\ldots,i_\kappa\}\subset\{1,\ldots,s\}$, we bound sums over $\mathscr{F}_{i_1}\cap\cdots\cap\mathscr{F}_{i_\kappa}$.  These intersections contribute no $N^s$ main term.

\begin{proposition}[General intersection]\label{prop:intersection}
Let $C_d>5$ be as in Proposition~\ref{prop:F1}.
Let $s\geq 2$ and let $2\leq\kappa\leq s$.
Let $\{i_1,\ldots,i_\kappa\}\subset\{1,\ldots,s\}$ be distinct.
For every polynomial $g$ of degree~$d$, all $N$ large enough depending on $d$ and~$s$, all $\delta$ with $0<\delta<1/8$,
and all integers $A$ in the window
\[
\delta^{-(C_d+2)}<A<\delta^{-(C_d+3)},
\]
at least one of the following holds:
\begin{enumerate}
    \item \[
    \big|\Sg{\mathscr{F}_{i_1}\cap\cdots\cap\mathscr{F}_{i_\kappa}}\big|
    \ll_{d,s} \delta^{1/2} N^s(\log N)^{O(s^2)}.
    \]
    \item There exists $q\in\mathbb{Z}\setminus\{0\}$ with $|q|\ll_d\delta^{-O_d(1)}$ such that $\|qg\|_{C^\infty[N]}\ll_d\delta^{-O_d(1)}$.
\end{enumerate}
\end{proposition}

Without loss of generality, we may assume that $i_1=1$, since otherwise we may relabel the $m$-variables.
We then split
\begin{align*}
\Sg{\mathscr{F}_{i_1}\cap\cdots\cap\mathscr{F}_{i_\kappa}}
&=\Sg{\mathscr{F}_{i_1}\cap\cdots\cap\mathscr{F}_{i_\kappa}\cap\{n_1=m_1\}}\\
&\qquad+\Sg{\mathscr{F}_{i_1}\cap\cdots\cap\mathscr{F}_{i_\kappa}\setminus\{n_1=m_1\}}\\
&=:I_2+I_1,
\end{align*}
as in the $J_1+J_2$ split of Proposition~\ref{prop:F1}.
In Section~\ref{sec:I1} we deal with the off-diagonal part $I_1$, and in Section~\ref{sec:I2} we deal with the diagonal part $I_2$.

\subsection{\texorpdfstring{Off-diagonal part $I_1$}{Off-diagonal part I_1}}
\label{sec:I1}

We now deal with $I_1$.
Write
\[
\mathscr{F}_I^\circ
:=\mathscr{F}_{i_1}\cap\cdots\cap\mathscr{F}_{i_\kappa}\setminus\{n_1=m_1\}
\]
for the off-diagonal part of the intersection (with $i_1=1$).

Similar to the $J_1$ case, we first need a parameterisation
lemma to extract the structure of arithmetic progressions.
But different from the $J_1$ case, the multiple intersections
of $\mathscr{F}_I^\circ$ have more constraints than a single $\mathscr{F}_1$, 
making it hard to directly apply the matrix parameterisation of Lemma~\ref{lem:param-general}.
The difficulty coming from the multiple constraints is that
the extra large gcd still involves the same $n_1$, and therefore cuts the fibre in $a_{11}$.

For example, take $i_1=1$ and $i_2=2$, that is $\mathscr{F}_1\cap\mathscr{F}_2$.
On $\mathscr{F}_1$, Lemma~\ref{lem:param-general} writes $n_1=a_{11}A_1$ and $m_1=a_{11}B_1$, with fibre $a_{11}\in[A,N/M_1]$.
Membership in $\mathscr{F}_2$ is $\gcd(n_1,m_2)\ge A$.
Equivalently, $\gcd(a_{11}A_1,m_2)\ge A$.
The admissible $a_{11}$ must share a large factor with $m_2$, so the fibre is no longer an interval, and Lemma~\ref{lem:diophantine-nil} does not apply.

To make this strategy work, we introduce new parameters $t_2,\dots ,t_s$
to separate the large gcds from the other factors. Hence the remaining
parameters are fully characterised by relatively coprime conditions, which makes it convenient to pass to exponential sums over arithmetic progressions.
More precisely, we have the following lemma:
\begin{lemma}[Intersection parameterisation]\label{lem:intersection-param}
Let $I=\{i_1,\ldots,i_\kappa\}\subseteq\{1,\ldots,s\}$ with $i_1=1$.
There is an injection
$\psi\colon\mathscr{F}_I^\circ\to\mathbb{Z}_{>0}^{3s}$ sending
$(n_1,\ldots,n_s,m_1,\ldots,m_s)$ to $(b,\nu)$, where
\[
\nu:=(t_2,\ldots,t_s,n_1',n_2,\ldots,n_s,m_1',m_2',\ldots,m_s').
\]
Denote $T:=t_2\cdots t_s$. The pair $(b,\nu)$ satisfies
\[
\begin{pmatrix}
n_1 & n_2 & \cdots & n_s \\
m_1 & m_2 & \cdots & m_s
\end{pmatrix}
=
\begin{pmatrix}
b T n_1' & n_2 & \cdots & n_s \\
b T m_1' & t_2 m_2' & \cdots & t_s m_s'
\end{pmatrix},
\]
together with
\[
\gcd(n_1',m_1')=1,\qquad
n_1'\neq m_1'.
\]
Write $\pi_2$ for the projection onto the outer parameters,
\[
\pi_2(b,\nu)=\nu,
\]
forgetting $b$.
For each $\nu\in\pi_2\bigl(\psi(\mathscr{F}_I^\circ)\bigr)$, we have
\[
\pi_2^{-1}(\nu)\cap\psi(\mathscr{F}_I^\circ)
=\mathscr G(\nu),
\]
where
\[
\mathscr G(\nu)
=\left\{(b,\nu):
\frac{A}{T}\leq b\leq
\frac{N}{T\max\{n_1',m_1'\}},\quad
\gcd(b,m_j')=1\ \text{for }2\leq j\leq s
\right\}.
\]
\end{lemma}
We take a moment to explain this lemma.
As in the $\mathscr{F}_1$ case we want to have a parameterisation
and then conduct exponential sums over the fibres, where the base points
can be grouped into good and bad parts.
The outer parameters \(\nu\) are the new base points, playing the role of \(\mathscr{F}_1'\).
The fibre variable \(b\) plays the role of \(a_{11}\).
The factors \(t_j\) are introduced to restore the information of the common divisors.
The lemma describes the fibres of $\pi_2$ on the image of $\psi$.
The fibre is exactly one interval satisfying some coprime conditions, where
the coprime conditions can be removed by M\"obius inversion in the following arguments.

\begin{proof}
We first construct $\psi$ and check injectivity, then verify the fibre description.

\smallskip\noindent\textbf{Construction.}
Set
\[
b_1=\gcd(n_1,m_1),\qquad
n_1'=\frac{n_1}{b_1},\qquad
m_1'=\frac{m_1}{b_1}.
\]
Then $\gcd(n_1',m_1')=1$. Recursively, for $j=2,\ldots,s$, define
\[
t_j=\gcd(b_{j-1},m_j),\qquad
b_j=\frac{b_{j-1}}{t_j},\qquad
m_j'=\frac{m_j}{t_j}.
\]
Write $b$ for $b_s$, and write $T$ for $t_2\cdots t_s$. Thus $b_1=bT$, so $n_1=bTn_1'$ and $m_1=bTm_1'$.
Note that from the construction, one immediately sees that $n_1'\neq m_1'$, that $\gcd(b_j,m_j')=1$ for $2\leq j\leq s$, and hence that $\gcd(b,m_j')=1$.
Also note that this procedure uniquely determines the parameterisation.

\smallskip\noindent\textbf{Injectivity.}
The map is injective because the pair is recovered as $n_1=bTn_1'$, $m_1=bTm_1'$, the coordinates $n_i$ for $i\geq 2$ as in $\nu$, and $m_j=t_jm_j'$ for $j\geq 2$.

\smallskip\noindent\textbf{Fibre.}
Fix a $\nu\in \pi_2\bigl(\psi(\mathscr{F}_I^\circ)\bigr)$. Write
\[
\nu=(t_2,\ldots,t_s,n_1',n_2,\ldots,n_s,m_1',m_2',\ldots,m_s'),
\]
and denote $T:=t_2\cdots t_s$.
Note that this $\nu$ arises from a construction.  So it naturally satisfies the properties $\gcd(n_1',m_1')=1$ and $n_1'\neq m_1'$ as noted in the construction.
We check the two inclusions between the fibre $\pi_2^{-1}(\nu)\cap\psi(\mathscr{F}_I^\circ)$ and $\mathscr G(\nu)$.

We first show that $\pi_2^{-1}(\nu)\cap\psi(\mathscr{F}_I^\circ)$ is contained in $\mathscr G(\nu)$.
Take an element of this fibre: it is of the form $(b,\nu)=\psi(n_1,\ldots,n_s,m_1,\ldots,m_s)$ for some $(n_1,\ldots,n_s,m_1,\ldots,m_s)\in\mathscr{F}_I^\circ$.
We first check the range of $b$.
Since $n_1=bTn_1'$ and $m_1=bTm_1'$, the bounds $n_1,m_1\leq N$ give
\[
b\leq\frac{N}{T\max\{n_1',m_1'\}}.
\]
The point lies in $\mathscr F_1$, so $\gcd(n_1,m_1)\geq A$.
From the construction this gcd equals $bT$, hence $b\geq A/T$.
The coprime conditions $\gcd(b,m_j')=1$ for $2\leq j\leq s$ likewise come from the construction.
Thus $(b,\nu)\in\mathscr G(\nu)$.

We next show that $\mathscr G(\nu)$ is contained in $\pi_2^{-1}(\nu)\cap\psi(\mathscr{F}_I^\circ)$.
Let $(b,\nu)$ be a point of $\mathscr G(\nu)$.
Our strategy is to reconstruct the integral vector $(n_1,\ldots,n_s,m_1,\ldots,m_s)$ from $(b,\nu)$ and check that it lies in $\mathscr{F}_I^\circ$ and that its image under $\psi$ is exactly $(b,\nu)$.

We reconstruct $n_1=bTn_1'$ and $m_1=bTm_1'$, together with $n_i$ as in $\nu$ for $i\geq 2$ and $m_j=t_j m_j'$ for $j\geq 2$.

We first check that the reconstructed integral vector $(n_1,\ldots,n_s,m_1,\ldots,m_s)$ lies in $\mathscr{F}_I^\circ$.
In order to check this, we must check that it lies in $V$, that it lies in every $\mathscr F_i$ for $i\in I$, and that $n_1\neq m_1$.
Note that $(b,\nu)\in \mathscr G(\nu)$, so $b$ satisfies
\[
\tfrac{A}{T}\leq b\leq\tfrac{N}{T\max\{n_1',m_1'\}}, \quad \gcd(b,m_j')=1\ \text{for }2\leq j\leq s.
\]

We first check membership in $V$.
The upper bound $b\leq N/(T\max\{n_1',m_1'\})$ gives $n_1=bTn_1'\leq N$ and $m_1=bTm_1'\leq N$. The coordinates $n_2,\ldots,n_s$ appear in $\nu$, hence are at most $N$. For $j\geq 2$, the reconstructed $m_j=t_j m_j'$ is the corresponding coordinate of the original point in $V$ that produced $\nu$, hence is also at most $N$. The coordinates are positive integers. The product identity  holds because $\nu$ arises from a point in the image of $\psi$, so $n_1'n_2\cdots n_s=m_1'(t_2m_2')\cdots(t_sm_s')$, and multiplying by $bT$ gives the identity $n_1\cdots n_s=m_1\cdots m_s$.

We next check membership in $\mathscr F_1$. Note that $\gcd(n_1',m_1')=1$, so $\gcd(n_1,m_1)=bT\geq A$.
We next check the condition $n_1\neq m_1$. As recorded above one has $n_1'\neq m_1'$, and multiplying by $bT\geq 1$ yields $n_1=bTn_1'\neq bTm_1'=m_1$.

We next check membership in each $\mathscr F_i$ for $i\in I\setminus\{1\}$.
For each such $i$, reconstruction gives $\gcd(n_1,m_i)=\gcd(bTn_1',t_im_i')$.
Since $t_i\mid T$, we may factor $t_i$ out of the gcd:
\[
\gcd(bTn_1',t_im_i')=t_i\gcd\bigl(b(T/t_i)n_1',m_i'\bigr).
\]
Since $\gcd(b,m_i')=1$, the remaining gcd drops $b$:
\[
\gcd\bigl(b(T/t_i)n_1',m_i'\bigr)=\gcd\bigl((T/t_i)n_1',m_i'\bigr).
\]
Hence $\gcd(n_1,m_i)=\gcd(Tn_1',t_im_i')$, which is independent of $b$.
Thus as $b$ varies in $\mathscr{G}(\nu)$, the reconstructed
\allowbreak$(n_1,\ldots,n_s,m_1,\ldots,m_s)$ has constant gcd.
Since $\nu\in \pi_2\bigl(\psi(\mathscr{F}_I^\circ)\bigr)$, there exists an element $\xi\in \mathscr{F}_I^\circ$ such that $\nu=\pi_2(\psi(\xi))$.
We may write $\psi(\xi)=(b_0,\nu)$.
By the first inclusion, $(b_0,\nu)$ is contained in $\mathscr G(\nu)$.
Since we have shown above that the reconstructed tuple has constant gcd throughout the fibre $\mathscr{G}(\nu)$,
we have that the gcds $\gcd(n_1,m_i)$ for $i\in I\setminus\{1\}$ are all the same as those of the original element $\xi$.
Thus the reconstructed tuple lies in $\mathscr{F}_I^\circ$.

Next we check that the image of
the reconstructed integral vector $(n_1,\ldots,n_s,m_1,\ldots,m_s)$ under $\psi$ is exactly $(b,\nu)$.
Since $\nu$ arises from the construction, one has $\gcd(n_1',m_1')=1$, so the first step recovers $b_1=bT$. Since the procedure uniquely determines the parameterisation, and $\gcd(b,m_j')=1$ for $2\leq j\leq s$, the successive steps recover the same $t_j$ and $m_j'$, and thus $\psi(n_1,\ldots,n_s,m_1,\ldots,m_s)=(b,\nu)$.
\end{proof}

We now deal with $I_1$. The following lemma is the goal of this section.

\begin{lemma}[Off-diagonal part $I_1$]\label{lem:I1}
Let the parameters be as in Proposition~\ref{prop:intersection}.
Recall
\[
I_1=\Sg{\mathscr{F}_I^\circ}.
\]
At least one of the following holds:
\begin{enumerate}
    \item $|I_1|\ll_{d,s} \delta^{1/2} N^s(\log N)^{O(s^2)}$.
    \item There exists $q\in\mathbb{Z}\setminus\{0\}$ with $|q|\ll_d\delta^{-O_d(1)}$ such that $\|qg\|_{C^\infty[N]}\ll_d\delta^{-O_d(1)}$.
\end{enumerate}
\end{lemma}

We now explain the proof strategy for Lemma~\ref{lem:I1}.
Write \(\mathscr{P}=\pi_2\bigl(\psi(\mathscr{F}_I^\circ)\bigr)\) for the finite set of outer parameters.
The strategy is illustrated in Figure~\ref{fig:I1}.
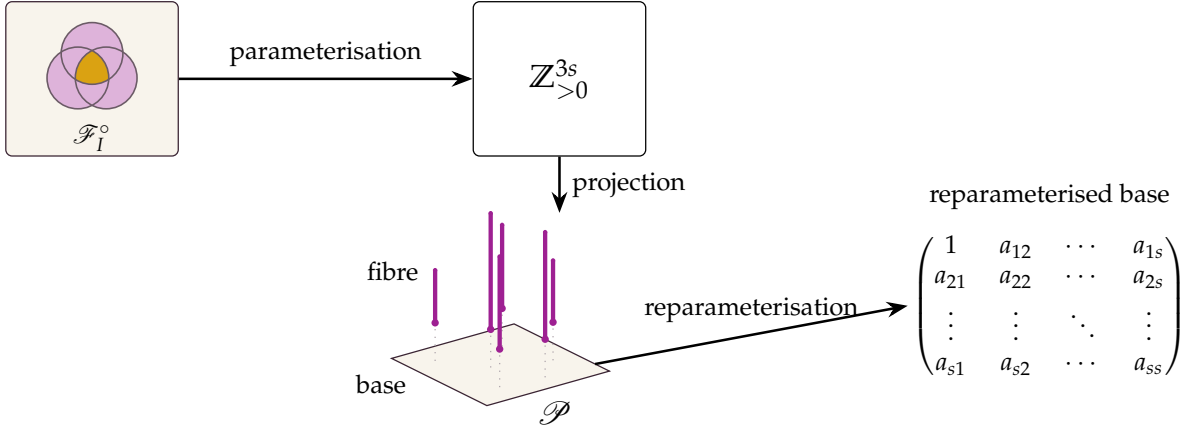
\begin{figure}[ht]
\centering
\resizebox{0.95\textwidth}{!}{%
\begin{tikzpicture}[
  font=\small,
  >=Stealth,
  arr/.style={->, thick},
]
  \node[draw=figink, rounded corners=2pt, fill=figplane,
        minimum width=1.9cm, minimum height=1.7cm]
    (FI) at (0,2.45) {};
  \coordinate (Ic) at ($(FI.center)+(0,0.10)$);
  \foreach \ang in {90,210,330} {
    \fill[plum!34] ($(Ic)+({0.18*cos(\ang)},{0.18*sin(\ang)})$) circle (0.34);
  }
  \begin{scope}
    \clip ($(Ic)+({0.18*cos(90)},{0.18*sin(90)})$) circle (0.34);
    \clip ($(Ic)+({0.18*cos(210)},{0.18*sin(210)})$) circle (0.34);
    \fill[gold] ($(Ic)+({0.18*cos(330)},{0.18*sin(330)})$) circle (0.34);
  \end{scope}
  \foreach \ang in {90,210,330} {
    \draw[figink!75, line width=0.5pt]
      ($(Ic)+({0.18*cos(\ang)},{0.18*sin(\ang)})$) circle (0.34);
  }
  \node[font=\scriptsize, anchor=north, inner sep=0pt]
    at ($(FI.center)+(0,-0.50)$) {$\mathscr{F}_I^\circ$};

  \node[draw, rounded corners=2pt, minimum width=1.9cm, minimum height=1.7cm]
    (Zbox) at (5.15,2.45) {$\mathbb{Z}_{>0}^{3s}$};

  \draw[arr] (FI.east) -- node[above, font=\scriptsize]{parameterisation} (Zbox.west);

  \coordinate (O) at (4.35,-1.15);
  \coordinate (U) at (1.35,0.38);
  \coordinate (V) at (-1.05,0.52);
  \def\ah{0.42}
  \fill[figplane] (O) -- ++(U) -- ++(V) -- ++($-1*(U)$) -- cycle;
  \draw[figink] (O) -- ++(U) -- ++(V) -- ++($-1*(U)$) -- cycle;
  \node[below=1pt] at ($(O)+0.55*(U)$) {$\mathscr{P}$};
  \node[font=\scriptsize, left=8pt] at ($(O)+0.45*(V)$) {base};
  \node[font=\scriptsize, anchor=east, inner sep=1pt]
    at ($(O)+0.08*(U)+0.78*(V)+(0,1.05)$) {fibre};

  \draw[arr] (Zbox.south) -- node[right, font=\scriptsize]{projection} ++(0,-0.62);

  \foreach \s/\t/\h in {
    0.20/0.80/0.58,
    0.70/0.74/0.92,
    0.42/0.50/1.28,
    0.82/0.36/0.68,
    0.26/0.20/1.02,
    0.60/0.16/1.18%
  } {
    \coordinate (P) at ($(O)+\s*(U)+\t*(V)$);
    \coordinate (Bot) at ($(P)+(0,\ah)$);
    \coordinate (Top) at ($(Bot)+(0,\h)$);
    \draw[figink!40, dotted] (P) -- (Bot);
    \draw[figblue, line width=1.35pt] (Bot) -- (Top);
    \fill[figblue] (Bot) circle (1.15pt);
    \fill[figblue] (Top) circle (0.7pt);
  }

  \node[inner sep=1pt, font=\scriptsize] (Mat) at (10.55,-0.05) {%
    $\begin{pmatrix}
    1 & a_{12} & \cdots & a_{1s} \\
    a_{21} & a_{22} & \cdots & a_{2s} \\
    \vdots & \vdots & \ddots & \vdots \\
    a_{s1} & a_{s2} & \cdots & a_{ss}
    \end{pmatrix}$};
  \node[above=2pt, font=\scriptsize] at (Mat.north) {reparameterised base};

  \draw[arr] ($(O)+(U)+0.15*(V)$) -- node[above, font=\scriptsize]{reparameterisation} (Mat.west);
\end{tikzpicture}%
}
\caption{The map \(\psi\) of Lemma~\ref{lem:intersection-param} sends \(\mathscr{F}_I^\circ\)
into \(\mathbb{Z}_{>0}^{3s}\).  The projection \(\pi_2\) forgets the first coordinate
 leaving the base \(\mathscr{P}\) with a fibre.
We group the fibres in the same spirit as in Lemma \ref{lem:J1}. A reparameterisation of the base is needed.}
\label{fig:I1}
\end{figure}

The estimate follows Figure~\ref{fig:I1} in four steps.
From parameterisation Lemma \ref{lem:intersection-param}, we may bound $I_1$
by first summing over the fibre and then summing over the base.
The idea is similar to the proof of Lemma \ref{lem:J1}, but we have to overcome
more difficulties.
The first difficulty is that the fibre is not an interval, but a more complicated set with coprime constraints.
By using M\"obius inversion, we may remove the coprime constraints by introducing a new parameter $\ell$.
The second difficulty is that although we want to classify the fibre into good and bad just like in Lemma \ref{lem:J1}, we have to face the case that the 
fibres have height too small to be classified when $\ell$ and $t_2\cdots t_s$ are large.
Hence we need to split the sum into three pieces, and estimate the first two pieces where the height is small by the divisor bounds of Lemma~\ref{lem:divisor}.
The third piece is where the height is large, and we need to reparameterise the base for this piece to reduce to Lemma~\ref{lem:J1}.
\begin{proof}
\smallskip\noindent\textbf{Step 1 (fibre sums).}
We start from
\[
I_1=\Sg{\mathscr{F}_I^\circ}
=\sum_{(n_1,\ldots,n_s,m_1,\ldots,m_s)\in\mathscr{F}_I^\circ}
\prod_{i=1}^s e\bigl(g(n_i)-g(m_i)\bigr).
\]
We use Lemma~\ref{lem:intersection-param} to parameterise the elements of $\mathscr{F}_I^\circ$.
Let $\mathscr P=\pi_2\bigl(\psi(\mathscr{F}_I^\circ)\bigr)$ be the finite set of outer parameters, as above. By Lemma~\ref{lem:intersection-param}, each $\nu\in\mathscr P$ has the form
\[
\nu=(t_2,\ldots,t_s,n_1',n_2,\ldots,n_s,m_1',m_2',\ldots,m_s').
\]
In Figure~\ref{fig:I1} this is the base after projection.
For $\nu\in\mathscr P$ put
\[
T(\nu):=t_2\cdots t_s,\qquad
Q(\nu):=m_2'\cdots m_s',\qquad
M(\nu):=\max\{n_1',m_1'\}.
\]
Lemma~\ref{lem:intersection-param}
gives the fibre with \(\gcd(b,m_j')=1\) for \(2\leq j\leq s\). Since \(\gcd(b,Q(\nu))=1\) if and only if \(\gcd(b,m_j')=1\) for \(2\leq j\leq s\), we may write
\[
\mathscr G(\nu)
=\Bigl\{(b,\nu):
\tfrac{A}{T(\nu)}\leq b\leq\tfrac{N}{T(\nu)M(\nu)},\quad
\gcd(b,Q(\nu))=1\Bigr\}.
\]
The identities of Lemma~\ref{lem:intersection-param} write the coordinates in terms of \((b,\nu)\), and the phase becomes
\[
\prod_{i=1}^s e\bigl(g(n_i)-g(m_i)\bigr)
=\Omega(\nu)\,e\bigl(g(bT(\nu)n_1')-g(bT(\nu)m_1')\bigr),
\]
where \(\Omega(\nu):=\prod_{i=2}^s e\bigl(g(n_i)-g(t_im_i')\bigr)\)
is unimodular and independent of $b$.
Substituting into \(I_1\) therefore gives
\begin{equation}\label{eq:I1-outer-decomposition}
I_1
=\sum_{\nu\in\mathscr P}\Omega(\nu)\,S(\nu),
\qquad
S(\nu)
:=\sum_{(b,\nu)\in\mathscr G(\nu)}
e\bigl(g(bT(\nu)n_1')-g(bT(\nu)m_1')\bigr).
\end{equation}
The inner sum $S(\nu)$ is the exponential sum over the fibre above the base point $\nu$.
The fibre is exactly an interval with some additional coprime constraints, and the coprimality
$\gcd(b,Q(\nu))=1$ is removed by M\"obius inversion by adding a new parameter in Step~2.

\smallskip\noindent\textbf{Step 2 (M\"obius inversion).}
In this step, our goal is to remove the coprime constraints by introducing a new parameter $\ell$.
Write
\[
\mathscr B(\nu)
:=\Bigl\{b\in\mathbb Z_{>0}:
\tfrac{A}{T(\nu)}\leq b\leq\tfrac{N}{T(\nu)M(\nu)}\Bigr\}
\]
for the interval without the coprimality constraint, so \(\mathscr G(\nu)=\bigl\{(b,\nu):b\in\mathscr B(\nu),\ \gcd(b,Q(\nu))=1\bigr\}\).
The coprimality indicator is
\[
\mathbf 1_{\gcd(b,Q(\nu))=1}
=\sum_{\ell\mid\gcd(b,Q(\nu))}\mu(\ell),
\]
where \(\ell\mid\gcd(b,Q(\nu))\) if and only if \(\ell\mid Q(\nu)\) and \(\ell\mid b\).
Inserting the identity into the constrained sum over \(\mathscr B(\nu)\) and changing the order of two finite sums therefore gives, for any function \(H\) on \(\mathbb Z_{>0}\),
\begin{equation}\label{eq:mobius-step}
\sum_{\substack{b\in\mathscr B(\nu)\\\gcd(b,Q(\nu))=1}}H(b)
=\sum_{\ell\mid Q(\nu)}\mu(\ell)
\sum_{\substack{k\geq1\\\ell k\in\mathscr B(\nu)}}H(\ell k).
\end{equation}
The fibre is rewritten as an interval in $k$ of height $N/(\ell T(\nu)M(\nu))$.
Apply \eqref{eq:mobius-step} with
$H(b)=e\bigl(g(bT(\nu)n_1')-g(bT(\nu)m_1')\bigr)$.  Then
\begin{equation}\label{eq:mobius-product}
S(\nu)=\sum_{\ell\mid Q(\nu)}\mu(\ell)\,S(\nu,\ell),
\end{equation}
where
\[
S(\nu,\ell)
:=\sum_{\substack{k\geq1\\\tfrac{A}{\ell T(\nu)}\leq k\leq\tfrac{N}{\ell T(\nu)M(\nu)}}}
e\bigl(g(\ell kT(\nu)n_1')-g(\ell kT(\nu)m_1')\bigr)
\]
Taking absolute values in
\eqref{eq:I1-outer-decomposition}--\eqref{eq:mobius-product}, and using
$|\mu(\ell)|\leq1$ together with $|\Omega(\nu)|=1$,
\begin{equation}\label{eq:I1-Mobius}
|I_1|
\leq\sum_{\nu\in\mathscr P}\sum_{\ell\mid Q(\nu)}\bigl|S(\nu,\ell)\bigr|.
\end{equation}
 
For the form of $S(\nu,\ell)$, we view it as the exponential sum over
a fibre, as drawn in Figure~\ref{fig:I1-short}.
In the same spirit as in the proof of Lemma \ref{lem:J1}, we want to classify these fibres into good and bad.
But when we look at the height of the fibre, that is, $N/\ell T(\nu)M(\nu)$, with $\ell$ and $\nu$
fixed.
These additional parameters $\ell$ and $T(\nu)$ will cause trouble when we try to 
copy the same argument as in Lemma~\ref{lem:J1}, since
when $\ell$ and $T(\nu)$ are large, the height of the fibre is too small to be an input of Lemma~\ref{lem:diophantine-nil}.
Our solution, as shown in the next step, is to isolate those large $\ell$ and $\nu$ with $T(\nu)$ large, and then apply the same argument as in Lemma~\ref{lem:J1} to bound the rest.

\smallskip\noindent\textbf{Step 3 (three-piece split).}
As noted above, we need to isolate those large $\ell$ and $\nu$ with $T(\nu)$ large, and then apply the same argument as in Lemma~\ref{lem:J1} to bound the rest.
Partition the pairs $(\nu,\ell)$ in \eqref{eq:I1-Mobius} according to the
size of $T(\nu)$ and $\ell$, and write
\begin{align*}
\Xi_{\mathrm I}
&:=\sum_{\substack{\nu\in\mathscr P\\ T(\nu)\geq\delta^{-1}}}
\sum_{\ell\mid Q(\nu)}\bigl|S(\nu,\ell)\bigr|,\\
\Xi_{\mathrm{II}}
&:=\sum_{\nu\in\mathscr P}
\sum_{\substack{\ell\mid Q(\nu)\\ \ell\geq\delta^{-1}}}
\bigl|S(\nu,\ell)\bigr|,\\
\Xi_{\mathrm{III}}
&:=\sum_{\substack{\nu\in\mathscr P\\ T(\nu)<\delta^{-1}}}
\sum_{\substack{\ell\mid Q(\nu)\\ \ell<\delta^{-1}}}
\bigl|S(\nu,\ell)\bigr|.
\end{align*}
Piece~I is the contribution of $\nu$ with large $T(\nu)$.
Piece~II is the contribution of large $\ell$.
Piece~III is what remains when both $T(\nu)$ and $\ell$ are small.
There is a harmless double counting when both $T(\nu)$ and $\ell$ are at
least $\delta^{-1}$.
Hence
\[
|I_1|
\leq
\Xi_{\mathrm I}+\Xi_{\mathrm{II}}+\Xi_{\mathrm{III}}.
\]
Pieces~I and~II are estimated by the divisor bounds of Lemma~\ref{lem:divisor} in this step;
The third piece contains the remaining base points with small $T(\nu)$ and $\ell$, where those base points
correspond to fibres of height large enough to be classified in the next step.

Figure~\ref{fig:I1-short} records this isolation.

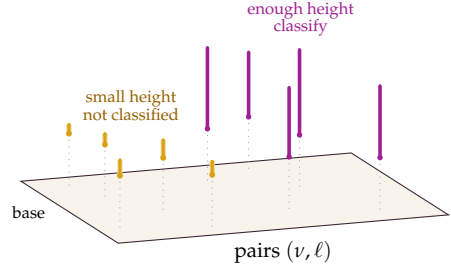
\begin{figure}[ht]
\centering
\begin{tikzpicture}[
  font=\small,
  scale=0.78,
  every node/.append style={transform shape},
]
  \coordinate (O) at (0,0);
  \coordinate (U) at (5.6,0.48);
  \coordinate (V) at (-1.65,1.05);
  \def\ah{0.88}

  \fill[figplane] (O) -- ++(U) -- ++(V) -- ++($-1*(U)$) -- cycle;
  \draw[figink] (O) -- ++(U) -- ++(V) -- ++($-1*(U)$) -- cycle;
  \node[below=2pt] at ($(O)+0.50*(U)$) {pairs $(\nu,\ell)$};

  \foreach \s/\t/\len in {
    0.07/0.22/0.26,
    0.16/0.68/0.18,
    0.26/0.42/0.30,
    0.11/0.88/0.15,
    0.32/0.12/0.22%
  } {
    \coordinate (P) at ($(O)+\s*(U)+\t*(V)$);
    \coordinate (Bot) at ($(P)+(0,\ah)$);
    \coordinate (Top) at ($(Bot)+(0,\len)$);
    \draw[figink!40, dotted] (P) -- (Bot);
    \draw[goldpale, line width=2pt] (Bot) -- (Top);
    \draw[gold, line width=1.15pt] (Bot) -- (Top);
    \fill[gold] (Bot) circle (1.4pt);
    \fill[gold] (Top) circle (0.9pt);
  }

  \foreach \s/\t/\len in {
    0.50/0.78/1.38,
    0.60/0.28/1.18,
    0.72/0.58/1.45,
    0.84/0.16/1.22,
    0.66/0.90/1.10%
  } {
    \coordinate (P) at ($(O)+\s*(U)+\t*(V)$);
    \coordinate (Bot) at ($(P)+(0,\ah)$);
    \coordinate (Top) at ($(Bot)+(0,\len)$);
    \draw[figink!40, dotted] (P) -- (Bot);
    \draw[figblue, line width=1.35pt] (Bot) -- (Top);
    \fill[figblue] (Bot) circle (1.3pt);
    \fill[figblue] (Top) circle (0.7pt);
  }

  \node[font=\scriptsize, align=center, text=goldink, anchor=south]
    at ($(O)+0.18*(U)+0.48*(V)+(0,\ah)+(0,0.48)$)
    {small height\\[-1pt]not classified};
  \node[font=\scriptsize, align=center, text=figblue, anchor=south]
    at ($(O)+0.70*(U)+0.50*(V)+(0,\ah)+(0,1.64)$)
    {enough height\\[-1pt]classify};
  \node[font=\scriptsize, left=8pt] at ($(O)+0.50*(V)$) {base};
\end{tikzpicture}
\caption{
 Each fibre is the floating interval.
Gold: small height, excluded from classification
(Pieces~I and~II).  Purple: enough height to enter the good/bad classification
of Lemma~\ref{lem:J1} (Piece~III).}
\label{fig:I1-short}
\end{figure}

Next, we deal with Piece~I and ~II.

\smallskip\noindent\textbf{Piece I ($T(\nu)\geq\delta^{-1}$).}
Recall that each $\nu\in\mathscr P$ is
\begin{align}\label{eq:nu-decomposition}
\nu=(t_2,\ldots,t_s,n_1',n_2,\ldots,n_s,m_1',m_2',\ldots,m_s'),
\end{align}
and write $T(\nu)=t_2\cdots t_s$, $Q(\nu)=m_2'\cdots m_s'$.

Piece I is dealt with by the divisor bounds of Lemma~\ref{lem:divisor}.
We do not use exponential sum bounds here.
The sum $S(\nu,\ell)$ runs over those $k$ with $\ell k\in\mathscr B(\nu)$, and each term has absolute value $1$. We use the trivial bound. Hence
\[
\Xi_{\mathrm I}
\leq
\sum_{\substack{\nu\in\mathscr P\\ T(\nu)\geq\delta^{-1}}}
\sum_{\ell\mid Q(\nu)}
\sum_{\substack{k\geq1\\ \ell k\in\mathscr B(\nu)}}1.
\]
We need to bound the right hand side of the above, while the
only budget we have is that the sum is over large $T(\nu)$.
Our strategy is to change the order of summation and count the number of $\nu$ with fixed $T,\ell,k$.
Changing the order, the right hand side of the above equals
\[
\sum_{T\geq\delta^{-1}}\sum_{\ell\geq1}\sum_{k\geq1}
\#\bigl\{\nu\in\mathscr P:\ T(\nu)=T,\ \ell\mid Q(\nu),\ \ell k\in\mathscr B(\nu)\bigr\}.
\]
The above is saying that we first fix $T,\ell,k$ and count the number of base points $\nu$ that satisfy the conditions.
The observation is that the number of $\nu$ with large $T(\nu)$ is well bounded.

\begin{clm}
Let $T,\ell,k\geq 1$.  Then
\begin{equation}\label{eq:pieces-divisor-majorant}
\#\bigl\{\nu\in\mathscr P:\ T(\nu)=T,\ \ell\mid Q(\nu),\ \ell k\in\mathscr B(\nu)\bigr\}
\leq
\frac{N^s}{k\ell^2T^2}\,
\tau_{s-1}(T)\tau_{s-1}(\ell)\tau_s(\ell T)(2s\log N)^{s^2-1}.
\end{equation}
\end{clm}

\noindent\textit{Proof of the claim.}
Recall that we write $\nu$ as in \eqref{eq:nu-decomposition}.
For fixed $T,\ell,k$,
we first observe the conditions that $\nu$ in the above set must satisfy, and then use the conditions to bound the number of $\nu$ with fixed $T,\ell,k$.
The conditions $T(\nu)=T$ and $\ell\mid Q(\nu)$ imply
\begin{equation}\label{eq:ooo}
n_1'n_2\cdots n_s
=T\,m_1'm_2'\cdots m_s',\qquad
\ell\mid m_2'\cdots m_s',
\end{equation}
while $\ell k\in\mathscr B(\nu)$ implies
$\ell kTn_1'\leq N$ and $\ell kTm_1'\leq N$.
Also note that $n_2,\ldots,n_s\leq N$. 
Gathering conditions we need, it therefore suffices to bound
\[
\#\bigl\{
(t_2,\ldots,t_s,n_1',n_2,\ldots,n_s,m_1',\ldots,m_s')
:\ 
t_2\cdots t_s=T,\ 
\text{\eqref{eq:ooo} holds},\ 
n_2,\ldots,n_s\leq N,\ 
\ell kTn_1'\leq N
\bigr\}.
\]

Our strategy is to first fix the following product
\[
n:=n_1'n_2\cdots n_s=T\,m_1'm_2'\cdots m_s',
\]
and bound the number of tuples for this $n$ by the divisor function, and then do the summation over all possible values of $n$.

We first observe the possible range of $n$.
Since \(n_2,\ldots,n_s\leq N\) and \(n_1'\leq N/(\ell kT)\), the possible $n$ satisfy \(n\leq N^s/(\ell kT)\).
From \(n=T\,m_1'm_2'\cdots m_s'\) and \(\ell\mid m_2'\cdots m_s'\), one has \(\ell T\mid n\).

Now fix such an $n$.
The tuple $(t_2,\ldots,t_s)$ is an ordered $(s-1)$-factorisation of the already fixed product $T$, hence at most $\tau_{s-1}(T)$ choices.
The condition $\ell\mid Q(\nu)$ does not specify how $\ell$ sits inside $m_2',\ldots,m_s'$.
Write $\ell=\ell_2\cdots\ell_s$ with $\ell_i\mid m_i'$ and $m_i'=\ell_i r_i$.
There are at most $\tau_{s-1}(\ell)$ such allocations.
After this substitution the identity $n=T\,m_1'm_2'\cdots m_s'$ becomes $n=\ell T\,m_1'r_2\cdots r_s$.
Thus $(m_1',r_2,\ldots,r_s)$ is an ordered $s$-factorisation of $n/(\ell T)$, hence at most $\tau_s(n/(\ell T))$ choices.
Independently, $(n_1',n_2,\ldots,n_s)$ is an ordered $s$-factorisation of $n$, hence at most $\tau_s(n)$ choices.
For this $n$, the number of tuples is therefore at most
$\tau_{s-1}(T)\,\tau_{s-1}(\ell)\,\tau_s(n)\,\tau_s\bigl(n/(\ell T)\bigr)$.

Summing over all possible values of $n$ gives
\[
\tau_{s-1}(T)\,\tau_{s-1}(\ell)
\sum_{\substack{n\leq N^s/(\ell kT)\\\ell T\mid n}}
\tau_s(n)\tau_s\!\left(\frac{n}{\ell T}\right).
\]
Write \(n=\ell Tm\) and use \(\tau_s(\ell Tm)\leq\tau_s(\ell T)\tau_s(m)\).  Then
\[
\sum_{\substack{n\leq N^s/(\ell kT)\\\ell T\mid n}}
\tau_s(n)\tau_s\!\left(\frac{n}{\ell T}\right)
\leq
\tau_s(\ell T)\sum_{m\leq N^s/(k\ell^2 T^2)}\tau_s(m)^2,
\]
and Lemma~\ref{lem:divisor} gives the asserted bound.
This completes the proof of the claim.
\hfill\(\square\)

\smallskip
Now sum \eqref{eq:pieces-divisor-majorant} over $T$, $\ell$ and $k$.
Here $T\leq N^{s-1}$ and, since $\ell\mid Q$, also $\ell\leq N^{s-1}$.
Using $\sum_{k\leq N}k^{-1}\ll\log N$ and
$\tau_s(\ell T)\leq\tau_s(\ell)\tau_s(T)$, $\Xi_{\mathrm I}$ is at most
\[
N^s(2s\log N)^{s^2}
\Biggl(
\sum_{\ell\leq N^{s-1}}
\frac{\tau_s(\ell)\tau_{s-1}(\ell)}{\ell^2}
\Biggr)
\Biggl(
\sum_{\delta^{-1}\leq T\leq N^{s-1}}
\frac{\tau_s(T)\tau_{s-1}(T)}{T^2}
\Biggr).
\]
Since $\tau_{s-1}\leq\tau_s$, Lemma~\ref{lem:divisor} applies to both sums.
The $\ell$-sum is a full sum and the $T$-sum is a tail from $\delta^{-1}$, so this is
\[
\ll_{d,s}\delta^{1/2} N^s(\log N)^{O(s^2)}.
\]

\smallskip\noindent\textbf{Piece II ($\ell\geq\delta^{-1}$).}
The bound~\eqref{eq:pieces-divisor-majorant} is symmetric in the final $T$- and $\ell$-sums.
This time the $T$-sum is unrestricted and Lemma~\ref{lem:divisor} is
applied to the tail $\ell\geq\delta^{-1}$.  Hence
\[
\Xi_{\mathrm{II}}
\ll_{d,s}\delta^{1/2} N^s(\log N)^{O(s^2)}.
\]

\smallskip\noindent\textbf{Step 4 (reparameterisation of the base; $T(\nu),\ell<\delta^{-1}$).}
It remains to bound $\Xi_{\mathrm{III}}$.
Both $T(\nu)$ and $\ell$ are now small. Hence our strategy is the same as in Lemma~\ref{lem:J1}, that is, to conduct exponential sums over the fibres,
group those base points as good or bad, and then apply the same argument as in Lemma~\ref{lem:J1} to bound the bad base points. In this step, a reparameterisation of the base is needed.
Recall that Piece~III is
\[
\Xi_{\mathrm{III}}
=\sum_{\substack{\nu\in\mathscr P\\ T(\nu)<\delta^{-1}}}
\sum_{\substack{\ell\mid Q(\nu)\\ \ell<\delta^{-1}}}
\bigl|S(\nu,\ell)\bigr|.
\]
Recall that each $\nu\in\mathscr P$ is a $(3s-1)$-dimensional integral vector
\[
\nu=(t_2,\ldots,t_s,n_1',n_2,\ldots,n_s,m_1',m_2',\ldots,m_s').
\]
For each $\nu$, we write $T(\nu)=t_2\cdots t_s$, $Q(\nu)=m_2'\cdots m_s'$, and $M(\nu)=\max\{n_1',m_1'\}$.
Note that for each $\nu$ fixed, and $\ell$ fixed, considering the inner sum $S(\nu,\ell)$, we have
\[
S(\nu,\ell)
=\sum_{\tfrac{A}{\ell T(\nu)}\leq k\leq\tfrac{N}{\ell T(\nu)\,M(\nu)}}
e\bigl(g(\ell k T(\nu)\,n_1')-g(\ell k T(\nu)\,m_1')\bigr).
\]
This is the desired exponential sum over an arithmetic progression.
Write \(h=\ell T(\nu)\). 
Our plan is to apply Lemma \ref{lem:diophantine-nil} to the polynomial $g_h(\cdot)$ defined by $g_h(\cdot)=g(\cdot h)$.
This requires the outer parameters $(t_2,\ldots,t_s)$ and $\ell$ to be fixed first.

Let $\pi_T(\nu)=(t_2,\ldots,t_s)$
be the projection of $\nu$ onto the first $s-1$ coordinates,
and let $\mathscr P_T:=\pi_T(\mathscr P)$ be the set of images of $\mathscr P$ under this projection.
As mentioned above, we may first fix the $(t_2,\ldots,t_s)$ coordinates and $\ell$
and sum over the remaining coordinates to use the same argument as in Lemma~\ref{lem:J1}.
So this requires us to change the order of summation.

Changing the order of summation,  and writing $\mathbf t=(t_2,\ldots,t_s)$, we obtain
\[
\Xi_{\mathrm{III}}
=\sum_{\substack{\mathbf t\in\mathscr P_T\\ t_2\cdots t_s<\delta^{-1}}}
\sum_{\ell<\delta^{-1}}
\sum_{\substack{\nu\in\mathscr P\\ \pi_T(\nu)=\mathbf t\\ \ell\mid Q(\nu)}}
\bigl|S(\nu,\ell)\bigr|.
\]
For $(\mathbf t,\ell)$ fixed, and hence $T=t_2\cdots t_s$ and $h=\ell T$ fixed, we consider the sum over the fibres and the base points $\nu$.
\begin{align}\label{eq:piece-III-sum1}
\sum_{\substack{\nu\in\mathscr P\\ \pi_T(\nu)=\mathbf t}}
\bigl|S(\nu,\ell)\bigr|.
\end{align}
Our plan is to  repeat the same argument as in Lemma~\ref{lem:J1} to bound the bad base points.
To get the same form as in Lemma~\ref{lem:J1}, we need to reparameterise the base.
For $\mathbf t=( t_2,\ldots, t_s)$ fixed, consider all $\nu=(t_2,\ldots,t_s,n_1',n_2,\ldots,n_s,m_1',m_2',\ldots,m_s')\in \mathscr P$.
We record the conditions that $\nu$ must satisfy in order to put the base in the same form as in Lemma~\ref{lem:param-general}:
\begin{center}
\begin{tabular}{@{}l@{\hspace{2.2em}}l@{}}
(i) $n_1'n_2\cdots n_s=m_1'(t_2m_2')\cdots(t_sm_s')$
&
(ii) $\gcd(n_1',m_1')=1$ and $n_1'\neq m_1'$
\\[0.45em]
(iii) $n_1',\ldots,n_s,m_1'\leq N$
&
(iv) $t_jm_j'\leq N$ for $2\leq j\leq s$.
\end{tabular}
\end{center}
Conditions~(i) and~(ii) come from Lemma \ref{lem:intersection-param}, and~(i) is the same product equation as in Remark~\ref{rem:param-remark}.
Conditions~(iii) and~(iv) are the bounds coming from $V$.
Every $\nu$ in~\eqref{eq:piece-III-sum1} satisfies (i)--(iv), so summing instead over all positive integers $n_1',\ldots,n_s,m_1',\ldots,m_s'$ obeying these conditions only enlarges the range:
\[
\eqref{eq:piece-III-sum1}
\leq
\sum_{\text{(i)--(iv)}}
\bigg|
\sum_{A/h\leq k\leq N/(h\max\{n_1',m_1'\})}
e\bigl(g(hkn_1')-g(hkm_1')\bigr)
\bigg|.
\]
The inner sum does not depend on $m_2',\ldots,m_s'$.
For $\mathbf t$ fixed, set $m_j=t_jm_j'$ when $2\leq j\leq s$, and rename $n_1',m_1'$ to $n_1,m_1$.
The product equation~(i) becomes $n_1\cdots n_s=m_1\cdots m_s$, with conditions (iii) and (iv) turned into $m_j,n_j\leq N$ when $1\leq j\leq s$ and $t_j\mid m_j$ when $2\leq j\leq s$.
Dropping the divisibility $t_j\mid m_j$ only enlarges the range. Applying~(ii), the previous display is at most the same sum over
\[
\mathscr{V}
:=\bigl\{
(\vec{n}_s,\vec{m}_s)\in V:\
\gcd(n_1,m_1)=1,\ n_1\neq m_1
\bigr\}.
\]
Thus the previous bound is at most
\[
\sum_{(\vec{n}_s,\vec{m}_s)\in\mathscr{V}}
\Biggl|
\sum_{\tfrac{A}{h}\leq k\leq\tfrac{N}{h\max\{n_1,m_1\}}}
e\bigl(g(hkn_1)-g(hkm_1)\bigr)
\Biggr|.
\]
By Remark~\ref{rem:param-remark} we may apply the map $\phi$ from Lemma~\ref{lem:param-general} to \(\mathscr{V}\) and get the parameterisation $(a_{ij})_{1\leq i,j\leq s}$ for each $(\vec{n}_s,\vec{m}_s)\in\mathscr{V}$.
Since \(\gcd(n_1,m_1)=1\) for every point of \(\mathscr{V}\), the construction gives \(a_{11}=1\).
As in Lemma \ref{lem:param-general}, write $A_1=\prod_{j=2}^s a_{1j}$, $B_1=\prod_{i=2}^s a_{i1}$, and $M_1=\max\{A_1,B_1\}$.
Since $a_{11}=1$, one has $n_1=A_1$, $m_1=B_1$, and $\max\{n_1,m_1\}=M_1$, so the summation above becomes
\[
\sum_{\phi(\mathscr{V})}
\Biggl|
\sum_{A/h\leq k\leq N/(hM_1)}
e\bigl(g(hkA_1)-g(hkB_1)\bigr)
\Biggr|.
\]
Moreover, the inner sum is constant along the remaining diagonal $(a_{22},\ldots,a_{ss})$, and that fibre has at most $\prod_{i=2}^s N/M_i$ points, the same trivial bound as in~\eqref{eq:J1-bound}.
Summing over all off-diagonal matrices therefore yields
\begin{equation}\label{eq:piece-III-matrix}
\sum_{\pi_1\bigl(\phi(\mathscr{V})\bigr)}
\Biggl|
\sum_{A/h\leq k\leq N/(hM_1)}
e\bigl(g(hkA_1)-g(hkB_1)\bigr)
\Biggr|
\prod_{i=2}^s\frac{N}{M_i}.
\end{equation}
This is the scaled form of~\eqref{eq:J1-bound}, where the main differences from \eqref{eq:J1-bound} are that the polynomial $g(\cdot)$ is replaced by $g(h\cdot)$ and the fibre range is divided by a factor $h$.

Apply the same arguments as in Lemma \ref{lem:J1} to
$g_h(n):=g(hn)$.  The summation range is $A/h\le k\le N/(hM_1)$.
The conditions required for the argument are guaranteed by $h=\ell T<\delta^{-2}$ and the restricted range of $A$.
Hence, unless alternative~(2) occurs for $g_h$, the argument of Claim~\ref{clm:J1-fibre} applies with $N$ and $A$ replaced by $N/h$ and $A/h$, respectively.
The same estimates for $\Sigma_{\mathrm I}$ and $\Sigma_{\mathrm{II}}$ as in Lemma~\ref{lem:J1} apply, with the factor $N/M_1$ replaced by $N/(hM_1)$.
Both sums therefore pick up an extra $1/h$.
Thus, unless the Diophantine alternative for $g_h$ occurs,
\begin{equation}\label{eq:piece-III-fixed}
\eqref{eq:piece-III-matrix}
\ll_{d,s} \delta^{1/2}\,\frac{N^s}{h}(\log N)^{O(s^2)}.
\end{equation}

Suppose alternative~(2) occurs for $g_h$.
Then there is a nonzero integer $K$ with $|K|\ll_d\delta^{-O_d(1)}$ and $\|Kg_h\|_{C^\infty[N/h]}\ll_d\delta^{-O_d(1)}$.
Write $g(n)=\sum_{j=0}^d\beta_j n^j$, so $g_h(n)=\sum_{j=0}^d(\beta_j h^j)n^j$.
Set $K'=Kh^d$.
Since $h=\ell T<\delta^{-2}$, one has $|K'|\ll_d\delta^{-O_d(1)}$.
By a simple calculation,
\[
\|K'g\|_{C^\infty[N]}
\leq h^d\|Kg_h\|_{C^\infty[N/h]}
\ll_d\delta^{-O_d(1)}.
\]
Thus $K'$ is the integer in alternative~(2) for $g$.

So far we have bounded the sum over the fibres and the base points $\nu$ with  $\mathbf t$ and $\ell$ fixed, that is, the quantity $\eqref{eq:piece-III-sum1}$. 
Finally, we need to bound the sum over all $\mathbf t$ and $\ell$.
The number of vectors $\mathbf t$ with product $T$ is at most
$\tau_{s-1}(T)$, while $\ell$ is already a scalar. 
Recall that $h=\ell T$.
Summing
\eqref{eq:piece-III-fixed} gives the additional factor
\[
\Biggl(\sum_{T<\delta^{-1}}\frac{\tau_{s-1}(T)}{T}\Biggr)
\Biggl(\sum_{\ell<\delta^{-1}}\frac1\ell\Biggr)
\ll (2\log N)^s
\]
by Lemma~\ref{lem:divisor}. 
  Multiplying the right hand side of~\eqref{eq:piece-III-fixed}
by $(2\log N)^s$ remains of size $\delta^{1/2}N^s(\log N)^{O(s^2)}$.   Hence
\[
\Xi_{\mathrm{III}}\ll_{d,s}\delta^{1/2} N^s(\log N)^{O(s^2)}.
\]

Combining the three estimates, we obtain
\[
|I_1|\ll_{d,s}\delta^{1/2} N^s(\log N)^{O(s^2)},
\]
unless Lemma~\ref{lem:diophantine-nil} produces alternative~(2).
\end{proof}

\subsection{\texorpdfstring{Diagonal part $I_2$}{Diagonal part I_2}}
\label{sec:I2}
We now deal with $I_2$. The following lemma is the goal of this section. 
\begin{lemma}[Diagonal $I_2$]\label{lem:I2}
Let the parameters be as in Proposition~\ref{prop:intersection}.
Recall that, with $i_1=1$,
\[
I_2=\Sg{\mathscr{F}_{i_1}\cap\cdots\cap\mathscr{F}_{i_\kappa}\cap\{n_1=m_1\}}.
\]
Then
\[
|I_2|\ll A^{-1}N^s(2\log N)^{2s^2-2}.
\]
In particular, $|I_2|\ll\delta N^s(2\log N)^{2s^2-2}\ll_{d,s}\delta^{1/2} N^s(\log N)^{O(s^2)}$ once $A>\delta^{-1}$.
\end{lemma}
\begin{proof}
Write $r=i_2\geq 2$.
Membership in $\mathscr{F}_1$ on the diagonal $n_1=m_1$ gives $n_1=m_1\geq A$.
Membership in $\mathscr{F}_r$ gives $\gcd(n_1,m_r)\geq A$. Since $n_1=m_1$, this is $\gcd(m_1,m_r)\geq A$.
Thus the support of $I_2$ sits inside
\[
\{n_1=m_1,\ \gcd(m_1,m_r)\geq A\}\cap V.
\]
Since $|e|=1$, it is enough to count the cardinality of the above set.

Write $n=n_1\cdots n_s=m_1\cdots m_s\leq N^s$. We first fix this $n$ and count the number of points satisfying the above conditions.
Fix $h=\gcd(m_1,m_r)$ with $h\geq A$.
Set $m_1=ha$ and $m_r=hb$ with $\gcd(a,b)=1$.
Then $h\mid m_1$ and $h\mid m_r$ contribute two copies of $h$ to the same factorisation of $n$, hence $h^2\mid n$.
The tuple $(a,b,(m_i)_{i\neq 1,r})$ is an ordered $s$-factorisation of
$n/h^2$, so the number of admissible $(m_1,\ldots,m_s)$ is at most $\tau_s(n/h^2)$.
The number of admissible $(n_1,\ldots,n_s)$ is at most $\tau_s(n)$.
Hence the contribution of this $h$ is $\sum_{n\leq N^s}\tau_s(n/h^2)\,\tau_s(n)$.
By Lemma~\ref{lem:divisor} and multiplicativity,
\[
\sum_{n\leq N^s}\tau_s(n/h^2)\,\tau_s(n)
\ll \frac{\tau_s(h)^2}{h^2}\,N^s(2\log N)^{s^2-1}.
\]
Summing over $h\geq A$ via $\sum_{h\geq A}\tau_s(h)^2/h^2\ll A^{-1}(2\log N)^{s^2-1}$ yields the claim.
\end{proof}

Combining Lemmas~\ref{lem:I1} and~\ref{lem:I2} proves Proposition~\ref{prop:intersection}.

\section{Conclusion of the proof}
\label{sec:conclusion}

\begin{proof}[Proof of Theorem~\ref{thm:main}]
We induct on \(s\geq 2\). The case \(s=2\) uses \(\mathcal{H}(1)\). For \(s\geq 3\), assume \(\mathcal{H}(s-1)\).

Let \(C_d>5\) be as in Proposition~\ref{prop:F1}, and choose \(c_0=c_0(d,s)>0\) sufficiently small.  Take \(N\) large enough depending on \(d\) and \(s\) that every estimate below applies.  Fix \(N^{-c_0}<\delta<1/8\) and an integer \(A\) with
\[
\delta^{-(C_d+2)}<A<\delta^{-(C_d+3)}.
\]
Let \(g\) be a polynomial of degree \(d\).

If alternative~(2) occurs in Proposition~\ref{prop:F1} or Proposition~\ref{prop:intersection}, the same $q$ gives alternative~(2) of the theorem.  Otherwise alternative~(1) holds in both, and the rest of the proof combines those bounds.

By Lemma~\ref{lem:moment-formula}, \(U_s(N)=\Sg{V}\).  Lemma~\ref{lem:sparse-complement} covers \(V\) by \(\bigcup_i\mathscr{F}_i\) up to a complement of size \(O\bigl(N^{s-1}A^{2s}(\log N)^{O(s^2)}\bigr)\).  The upper bound \(A<\delta^{-(C_d+3)}\) together with \(\delta>N^{-c_0}\) and \(c_0\) small puts this complement inside \(\delta^{1/2} N^s(\log N)^{O(s^2)}\).

Expand the union by~\eqref{eq:IE-general}.  Proposition~\ref{prop:F1} and Remark~\ref{rem:Fi-symmetry} give
\[
\Sg{\mathscr{F}_i}=(s-1)!\,N^s+O_{d,s}\bigl(\delta^{1/2} N^s(\log N)^{O(s^2)}\bigr)
\]
for each \(i\).  For \(\kappa\geq 2\), Proposition~\ref{prop:intersection} bounds each \(\kappa\)-fold intersection by \(O_{d,s}\bigl(\delta^{1/2} N^s(\log N)^{O(s^2)}\bigr)\).

There are \(s\) single sets and at most \(2^s\) intersection terms. Absorbing these factors into the implied constant,
\[
U_s(N)=s!\,N^s+O_{d,s}\bigl(\delta^{1/2} N^s(\log N)^{O(s^2)}\bigr).
\]  Dividing by \(N^s\) gives
\[
\bigl|\mathcal{M}_s(N)-s!\bigr|
\ll_{d,s}\delta^{1/2}(\log N)^{O(s^2)}.
\]
Replacing the free parameter $\delta$ by $\delta^2$ (and shrinking $c_0$) yields $\mathcal{H}(s)$.  The Diophantine alternative becomes $|q|\ll_d(\delta^2)^{-O_d(1)}=\delta^{-O_d(1)}$, as stated.
\end{proof}

\begin{proof}[Proof of Corollary~\ref{thm:cor-clt}]
Fix an integer $s\geq 2$ and set
\begin{equation}\label{eq:delta-clt}
\delta=\delta_s(N)=(\log N)^{-K},
\end{equation}
where $K=K(s)$ is large enough depending only on $s$ that the error $\delta(\log N)^{O(s^2)}$ in Theorem~\ref{thm:main} tends to $0$.
For $N$ large enough (depending on $d$ and~$s$), one has $\delta>N^{-c_0}$ with $c_0=c_0(d,s)$ as in Theorem~\ref{thm:main}, so the theorem applies.  Suppose towards a contradiction that alternative~(2) holds for this~$\delta$: there exists $q\in\mathbb{Z}\setminus\{0\}$ with
\[
|q|\ll_d\delta^{-O_d(1)}\qquad\text{and}\qquad
\|qg\|_{C^\infty[N]}\ll_d\delta^{-O_d(1)}.
\]
By Definition~\ref{def:c-infty}, the second bound means that for every $1\leq i\leq d$,
\[
\|q\beta_i\|_{\mathbb{R}/\mathbb{Z}}\ll_d\delta^{-O_d(1)}\,N^{-i}.
\]
Since $|q|\ll_d\delta^{-O_d(1)}$ and $\delta=(\log N)^{-K}$ with $K=K(s)$, one has $|q|\ll_{d,s}(\log N)^{O_{d,s}(1)}$.
Choose $\varepsilon=\varepsilon(d,s)>0$ small enough that $|q|^{\varepsilon}\le(\log N)/2$.
The hypothesis then yields $C=C(g,\varepsilon)>0$ such that some $i$ has
$\|q\beta_i\|\geq C\exp\bigl(-|q|^{\varepsilon}\bigr)$.  Combining with the display above,
\[
\exp\bigl(-|q|^{\varepsilon}\bigr)\ll_d\delta^{-O_d(1)}N^{-i}.
\]
The left side is at least $N^{-1/2}$, while the right side is
$\ll_{d,s}(\log N)^{O_{d,s}(1)}N^{-1}$.
This is impossible for large~$N$.  Thus alternative~(2) cannot occur, and Theorem~\ref{thm:main} gives
\[
\bigl|\mathbb{E}|S_N|^{2s}-s!\bigr|
\ll_{d,s}\delta(\log N)^{O(s^2)}
\to 0
\]
as $N\to\infty$.  Thus $\mathcal{M}_s(N)=\mathbb{E}|S_N|^{2s}\to s!$ for every fixed integer $s\geq 2$.  (The case $s=1$ is immediate: $\mathbb{E}|S_N|^2=1$ by Steinhaus orthogonality and $|e(g(n))|=1$.)

For odd moments, fix integers $s_1,s_2\geq 0$ with $s_1\neq s_2$; by conjugation, we may assume $s_2<s_1$.  If $s_2=0$, then the Steinhaus orthogonality gives
\[
\mathbb{E}\bigl[S_N^{s_1}\bigr]
=N^{-s_1/2}e\bigl(s_1g(1)\bigr)\to 0.
\]
We may therefore assume that $1\leq s_2<s_1$.  Expanding and using complete multiplicativity together with the Steinhaus orthogonality relation, we obtain
\begin{align*}
\mathbb{E}\bigl[S_N^{s_1}\,\overline{S_N}^{s_2}\bigr]
&=N^{-(s_1+s_2)/2}
\sum_{\substack{n_1,\ldots,n_{s_1}\leq N\\ m_1,\ldots,m_{s_2}\leq N\\ n_1\cdots n_{s_1}=m_1\cdots m_{s_2}}}
\prod_{i=1}^{s_1}e(g(n_i))
\prod_{j=1}^{s_2}\overline{e(g(m_j))}.
\end{align*}
Since $|e(g(\cdot))|=1$, the absolute value of the right-hand side is at most
\[
N^{-(s_1+s_2)/2}
\cdot\#\bigl\{(n_1,\ldots,n_{s_1},m_1,\ldots,m_{s_2})\in[1,N]^{s_1+s_2}:
n_1\cdots n_{s_1}=m_1\cdots m_{s_2}\bigr\}.
\]
Write $P=n_1\cdots n_{s_1}=m_1\cdots m_{s_2}$.  Then $P\leq N^{s_2}$, and for each such $P$ the number of ordered factorisations is at most $\tau_{s_1}(P)\tau_{s_2}(P)$.  Hence the cardinality is
\[
\leq\sum_{P\leq N^{s_2}}\tau_{s_1}(P)\tau_{s_2}(P).
\]
Cauchy--Schwarz and Lemma~\ref{lem:divisor} give
\begin{align*}
\sum_{P\leq N^{s_2}}\tau_{s_1}(P)\tau_{s_2}(P)
&\leq\Bigl(\sum_{P\leq N^{s_2}}\tau_{s_1}(P)^2\Bigr)^{1/2}
\Bigl(\sum_{P\leq N^{s_2}}\tau_{s_2}(P)^2\Bigr)^{1/2}\\
&\leq N^{s_2}\,(2\log N)^{O((s_1+s_2)^2)}.
\end{align*}
Therefore
\[
\bigl|\mathbb{E}\bigl[S_N^{s_1}\,\overline{S_N}^{s_2}\bigr]\bigr|
\ll_{s_1,s_2} N^{(s_2-s_1)/2}\,(\log N)^{O((s_1+s_2)^2)}
\to 0
\qquad(N\to\infty).
\]
Thus \(\mathbb{E}\bigl[S_N^{s_1}\,\overline{S_N}^{s_2}\bigr]\) tends to \(0\) whenever \(s_1\neq s_2\). Together with \(\mathbb{E}|S_N|^{2s}\) tending to \(s!\) for every \(s\geq 1\), the moments of \(S_N\) match those of a standard complex normal with mean \(0\) and variance \(1\). By the moment method \cite[Chapter~5, Theorem~8.6]{Gut05}, \(S_N\) converges in law to this distribution.
\end{proof}